\documentclass[12pt]{amsart}
\usepackage{amssymb}
\usepackage{mathrsfs}
\usepackage{amsmath,amstext,amsfonts,verbatim}
\usepackage{color}
\usepackage[latin1]{inputenc}
\usepackage{amsmath,amsthm,amssymb}
\newtheorem{thm}{Theorem}
\numberwithin{thm}{section}
\newtheorem{lemma}[thm]{Lemma}
\newtheorem{cor}[thm]{Corollary}
\newtheorem{prop}[thm]{Proposition}

\newcommand{\neweq}[1]{\begin{equation}\label{#1}}

\def\phi{\varphi}

\def\incep{\left\{\begin{array}{cl} }
 \def\termin{\end{array}\right. }
\def\2af{2^*_\alpha}

\def\dd{\mathbb{D}}

\title{ \textbf{Toeplitz operators on weighted harmonic Bergman spaces}}

\author{Zipeng Wang}
\address{College of Mathematics and Information Science, Xi'an, Shaanxi Normal University, 710062, PR \ China.}
\email{zipengwang@snnu.edu.cn}

\author{ Xianfeng  Zhao}
\address{ College of Mathematics and Statistics, Chongqing University, Chongqing, 401331, PR China and \ Shanghai \ Center\ for \ Mathematical \ Sciences,\ Shanghai, Fudan University, 200433,  PR \ China.}
\email{xianfengzhao@cqu.edu.cn}

\thanks{MR(2010) Subject Classification  47B 35, 47B 65}
\thanks{This work was partially supported by a NSFC grant(11271387) and Chongqing Natural Sience Foundation (cstc2013jjB0050). }

\begin{document}
\maketitle
\begin{center}{}\end{center}

\begin{abstract}
In this paper, we study Toeplitz operators on the  weighted  harmonic Bergman spaces  with nonnegative symbols, the weights we choose here are Muckenhoupt $\mathcal A_2$ weights.  Results obtained include characterizations of
 bounded Toeplitz operators, compact Toeplitz
operators, invertible Toeplitz operators  and Toeplitz operators in the Schatten  classes.
\end{abstract}

\section{Introduction}
Let $1 \leqslant p<\infty$ and $\omega$ be a nonnegative integrable function  on the unit disk $\dd$.  $L^p(\omega)$ denotes the Banach space with norm
\begin{equation*}
\|f\|_{L^p(\omega)}:=\bigg(\int_\dd |f(z)|^p \omega(z) dA(z)\bigg)^\frac{1}{p}.
\end{equation*}
The weighted harmonic (analytic) Bergman space $L^p_h(\omega)$ ($L^p_a(\omega)$)  is the subspace of $L^p(\omega)$ which is consisting of  harmonic (analytic) functions on $\mathbb D$. The goal of this paper is to provide a framework to study operator properties (boundedness, compactness, Schatten classes and inveritibility) of Toeplitz operators on $L_h^2(\omega)$ with nonnegative symbols.

Weighted analytic function spaces and their Toeplitz operators have captured people's attentions for a long time. It is now well known (\cite{Zhu2007}) that several results on unweighted Bergman space can be extended to the standard weighted Bergman space $L_a^2(\omega_\alpha)$,  where $\omega_\alpha(z)=(1+\alpha)(1-|z|^2)^\alpha$ and $-1<\alpha<\infty$.
In recent papers \cite{PR2015} and \cite{PR2016}, Pel\'{a}ez and R\"{a}tty\"{a} characterized the bounded and Schatten class Toeplitz operators (induced by a positive Borel measure) on a weighted Bergman space, here the weight is a radial function satisfying the doubling property $\int_r^1 \omega(s)ds \leqslant C\int_{\frac{1+r}{2}}^1 \omega(s)ds.$  The first results of non-radial weighted Bergman space seems to be due to Luecking (\cite{Luecking1985}) who investigated  the structure of weighted Bergman space with B\'{e}koll\'{e}-Bonami weight. Based on Luecking's representation and duality theorems in \cite{Luecking1985},  Chac\'{o}n (\cite{Cha}) and Constantin  (\cite{Con2007}, \cite{Con}) studied the boundedness and compactness of Toeplitz operators on certain weighted Bergman spaces. In \cite{MW2014}, Mitkovski and Wick established a reproducing kernel thesis for operators on Bergman type space, and their definitions include weighted versions of Bergman spaces on more complicated domains.

We will be primarily interested in weighted harmonic Bergman space $L_h^2(\omega)$.  Our choice of the weight $\omega$ is motivated by the  characterization of boundedness of $P_h$ acting on $L^2(\omega)$, where $P_h$ is the unweghted harmonic Bergman (orthogonal) projection from $L^2(dA)$ to $L_h^2(dA)$. It is well known that  $L_h^2(dA)$ is a reproducing kernel Hilbert space and
\begin{equation*}
P_hf(z)=\int_\dd f(\lambda) \bigg[\frac{1}{(1-\overline{\lambda }z)^2}+\frac{1}{(1-\overline{z}\lambda)^2}-1\bigg]dA(\lambda).
\end{equation*}

Clearly, $P_h$ is a Calder\'{o}n-Zygmund operator on the homogeneous space $(\dd,d,dA)$,  where $d$ is the Euclidean distance and $dA$ is the Lebesgue area measure on  $\mathbb D$, normalized so that the measure of $\mathbb D$ is $1$.
 For the  definitions of Calder\'{o}n-Zygmund operator and homogeneous space, we refer to \cite{AV2014}.

The most successful understanding of the (one) weight theory of Calder\'{o}n-Zgumund operator was spurred by Munckenhoupt's work in 1970s (\cite{MW1974}), which led to the introduction of the class of $\mathcal{A}_p$ weight and developments of weighted inequality. We will restrict our attentions to $\mathcal{A}_2$ weight on $(\dd, d, dA)$.

Let $0<\omega\in L^1(\dd, dA)$, it is called a Muckenhoupt $\mathcal{A}_2$ weight if
$$[\omega]_{\mathcal{A}_2}:=\sup_{a\in\dd,\  0<r<1} \frac{|B(a,r)|_{\omega}|B(a,r)|_{\omega^{-1}}}{|B(a, r)|^2} <+\infty,$$
where $$B(a,r)=\{z\in\dd: d(a,z)=|z-a|<r\},$$
$$|B(a,r)|_\omega=\int_{B(a, r)}\omega(z) dA(z)$$ and $|\cdot|$ is the normalized Lebesgue measure on $\dd$.

It follows form the remarkable $\mathcal{A}_2$ theorem (\cite{AV2014},\cite{Hy2012}) that \emph{$P_h$ is  bounded  from $L^2(\omega)$ to $L_h^2(\omega)$ provided $\omega$ is a Muckenhoupt $\mathcal{A}_2$ weight.}  As mentioned above, we will focus on  the weighted harmonic Bergman space $L_h^2(\omega)$ with  $\omega\in \mathcal{A}_2$.  Little is known about this nature function  space. However,  we will see in Section 2  that $L_h^2(\omega)$ is  a reproducing kernel Hilbert space with  the reproducing kernel  $K^\omega_z(\lambda)$, i.e.,
$f(z)=\langle f,K^\omega_z\rangle_{L^2(\omega)}$ for all  $f$  in $L_h^2(\omega)$.

For a positive finite Borel measure $\nu$ on $\mathbb D$,  we densely define the Toeplitz operator $T_\nu$ on  $L_h^2(\omega)$ by
 $$T_\nu f(z)=\big\langle T_\nu f, K^{\omega}_z \big\rangle_{L^2(\omega)}=\int_{\mathbb D}f(\lambda)\overline{K^{\omega}_z(\lambda)}d\nu(\lambda) \ \ \  \big(z\in \mathbb D\big).$$
For a bounded function $\varphi$, using the integral representation for the projection operator (from $L^2(\omega)$ to $ L_h^2(\omega)$), we can express the Toeplitz operator $T_\varphi$
(on $L_h^2(\omega)$) as the following:
$$T_\varphi f (z)=\int_{\mathbb D}f(\lambda) \overline{K^{\omega}_z(\lambda)} \varphi(\lambda)  \omega(\lambda)dA(\lambda) \ \  \ \big(z\in \mathbb D\big).$$

Although we follow Luecking's methods in \cite{Luecking1985} and \cite{Luecking AJM} for the weighted Bergman spaces, some new difficulties arise in the study of the space $L_h^2(\omega)$ and the corresponding operators. For instance,  harmonic
functions do not share many powerful tools with analytic functions. One can use the Cauchy formula to estimate the local values of  analytic functions easily. However, because of the tedious remainder, the harmonic version Cauchy formula (known as Cauchy-Pompeiu formula) is not valid now.
We instead must rely on some known estimations  on  harmonic functions.  In addition,  just as the weighted Bergman space, one cannot write down an explicit formula for  the reproducing kernel for $L_h^2(\omega)$. To overcome this obstacle, we will use the  reproducing kernel for the unweighted space $L_h^2$ to help us  study the representation theory of $L_h^2(\omega)$. However, the properties of reproducing kernel for $L_h^2$ are much more complicated than that of Bergman space $L_a^2$.

Using some properties of harmonic functions and $\mathcal {A}_2$ weights, we establish two different atomic decompositions for functions in $L_h^2(\omega)$ (Theorems \ref{Atomic Decomposition} and \ref{Atomic Decomposition cor}), which extend the representation theorems in \cite{Luecking1985} to the harmonic case.

In Section 3,  we  characterize the boundedness, compactness and  Schatten $p$ class of Toeplitz operators $T_\nu$ on $L_h^2(\omega)$  by means of  Berezin transform and Carleson measure. We are pleasure to  mention here that Miao (\cite{Miao}) have obtained characterizations for the Toeplitz operators with nonnegative symbols to be bounded, compact and in Schatten classes on unweighted harmonic Bergman space $L_h^2$.

Section 4 of this paper is devoted to studying the invertibility of Toeplitz operators on the standard weighted  harmonic Bergman space $L^2_h(\omega_\alpha)$.  A little surprising to us, the results illustrate that  the invertibility of Toeplitz operators on Bergman space $L_a^2$ can imply a reverse Carleson  inequality for $L_h^2(\omega_\alpha)$, see  Theorems \ref{invertibility1} and  \ref{invertibility2}. Based on this inequality, we  generalize  the result on the invertibility of  Bergman Toeplitz operators  with nonnegative symbols (\cite{Zhao}) to the case of $L_h^2(\omega_\alpha)$. As a consequence, we obtain a relationship of the invertibility  between  Toeplitz operators with bounded nonnegative symbols on $L_a^2(\omega_\alpha)$ and $L_h^2(\omega_\alpha)$, see Corollary \ref{invertibility3}.

In the final section, we establish a reverse Carleson type inequality for $L_h^2(\omega)$ with $\omega\in\mathcal{A}_2$. Indeed, we obtain a sufficient condition for $\chi_G dA$ to be a reverse Carleson measure for $L_h^2(\omega)$,  where $G$ is a measurable set in $\mathbb D$, see Theorem \ref{the weight w}, which  extends Theorem 3.9 for the weighted  (analytic) Bergman space in \cite{Luecking1985} to the harmonic setting.

Throughout the paper, positive constants will be explicitly denoted by $C, C_0, C_1,\cdots$, which may depend on some fixed numbers and
change at each occurrence.

\section{The Space $L_h^2(\omega)$ and its Representation}

In this section, we present some elementary structures of $L_h^2(\omega)$ with $\omega \in \mathcal A_2$.  To study the
harmonic Bergman spaces, we need the following important properties of harmonic functions.
\begin{lemma}\label{subharmonic behavior of harmonic functions} (\cite{Kuran})
Suppose that $f$ is a harmonic function on the disk $\mathbb D$ and $0<p<\infty$. There exists a positive constant
$C=C(p)$ such that for every ball  $B(a, r)=\{z\in \mathbb D: |z-a|<r\}$ in $\mathbb D$,
$$|f(a)|^p\leqslant \frac{C}{|B(a,r)|} \int_{B(a,r)} |f(z)|^p dA(z).$$
 In particular, if $p\geqslant 1$, then the constant $C\equiv 1$.
Using this result one can get the following useful inequalities easily: given $0<p<\infty$ and $0<r<1$, there exist  positive constants
$C_1=C_1(p)$ and $C_2=C_2(p)$  such that
$$|f(a)|^p \leqslant \frac{C_1}{(1-r)^{2}}\cdot \frac{1}{|D(a,r)|} \int_{D(a,r)} |f(z)|^p dA(z) \ \  \  \big(a \in \mathbb D\big)$$
and
$$|\partial f(a)|^p \leqslant \frac{C_2}{(1-r)^{2+p}}\cdot \frac{1}{|D(a,r)|^{\frac{p+2}{2}}} \int_{D(a,r)} |f(z)|^p dA(z) \ \  \  \big(a \in \mathbb D\big)$$
for all $f$ harmonic on $\mathbb D$, where $\partial f= \frac{\partial f}{\partial z}$.
\end{lemma}

\emph{Remark 1. }From the above inequalities, it is easy to  show that point evaluations are bounded linear functionals on $L_h^2(\omega)$ with $\omega \in \mathcal A_2$. As a consequence, $L_h^2(\omega)$ is a reproducing kernel  Hilbert space.

\emph{Remark 2.} It is not clear whether  $L^2_h(\omega)$ is complete for a general weight. However, if  $p$ is an analytic polynomial on $\mathbb D$ and  $\omega(z)=|p(z)|^2$, Douglas and Wang \cite{DW2011} showed  that $L_a^2(\omega)$ is complete, whose proof heavily depends on some  particular properties of polynomials.

It is clear that $L_h^2(\omega)$ coincides with its dual space with respect to the $L^2(\omega)$ inner product.  The next result illustrates that the dual space of $L_h^2(\omega)$ can be identified with $L^2_h(\omega^{-1})$ via the unweighted inner product,  which generalizes Luecking's result for $L_a^2(\omega)$ (see Theorem 2.1 in \cite{Luecking1985}) to the  setting  of $L_h^2(\omega)$.
\begin{lemma}\label{duality}
  Suppose that $\omega$ is an $\mathcal{A}_2$ weight. Then the dual space of $L_h^2(\omega)$ can be identified with $L_h^2(\omega^{-1})$. The pairing is given by
$$\langle f, g\rangle=\int_{\mathbb D} f(z)\overline{g(z)}dA(z).$$
Consequently, there exists a bounded,  bijective and linear operator $\mathscr{F}: L_h^2(\omega^{-1})\rightarrow (L_h^2(\omega))^*$ such that
$$\mathscr{F}(f)(g)=\langle g, \overline{f}\rangle_{L^2(dA)}$$
for $f\in L_h^2(\omega^{-1})$ and $g\in L_h^2(\omega)$.
\end{lemma}
\begin{proof}  Let $\omega$ be an $\mathcal {A}_2$ weight. Recall that orthogonal projection $P_h:L^2(dA)\to L^2_h(dA)$ is a Calder\'{o}n-Zygmund operator on $(\dd,d,dA)$. Then $P_h$ is bounded on both $L^2(\omega)$ and $L^2(\omega^{-1})$.
Thus for  each  $f\in L^2(\omega)$ and $g\in L^2(\omega^{-1})$, we have
$$\langle P_h f , g\rangle=\langle f, P_h g \rangle.$$
Note that each $f\in L_h^2(\omega)$  (or $f\in L_h^2(\omega^{-1})$) has the following form:
$$f(z)=\sum_{n=0}^\infty a_nz^n +\sum_{n=1}^\infty b_n \overline{z^n},$$
thus we obtain
\begin{align*}
P_hf(z)&= \int_{\mathbb D} f(\lambda)R_z(\lambda)dA(\lambda)\\
&=\lim_{s\rightarrow 1^{-}}\int_{s\mathbb D} f(\lambda)\Big[\frac{1}{(1-\overline{z}\lambda)^2}+\frac{1}{(1-\overline{\lambda}z)^2}-1\Big]dA(\lambda)\\
&=f(z)
\end{align*}
for $z\in \mathbb D$.
Then the rest of this proof are exactly the same as the one of Theorem 2.1 in \cite{Luecking1985}, we omit the details.
\end{proof}

Let $a\in \dd$ and $0<r<1$. A pseudohyperbolic disk $D(a, r)$ is defined by
\begin{equation*}D(a, r)=\bigg\{z\in \mathbb D: \rho(z, a)=\Big|\frac{z-a}{1-\overline{a}z}\Big|<r\bigg\}.
\end{equation*}

We will frequently use the following  property of $\mathcal{A}_2$ weights on pseudohyperbolic disks. For the sake of complete, we include a proof of this fact as follows.

\begin{lemma}\label{two disks}
Let $0<r \leqslant \frac{1}{4}$ and $z\in \mathbb D$. If $\xi\in D(z,r)$, then we have
\begin{equation*}
|D(z,r)|_\omega < 8 [\omega]_{\mathcal{A}_2} |D(\xi,r)|_\omega.
\end{equation*}
\end{lemma}
\begin{proof} Observe that $D(z, r) \subset D(\xi, 2r)$. Now it suffices to show the following doubling inequality:
\begin{equation*}\label{E:doubling}
|D(\xi, 2 r)|_\omega < 8 [\omega]_{\mathcal{A}_2}|D(\xi,r)|_\omega \ \ \   \big(\xi \in \mathbb D \big).
\end{equation*}

Since $\omega$ is an $\mathcal {A}_2$weight, we have
\begin{equation*}
\frac{|D(\xi, 2r)|_\omega|D(\xi,2r)|_{\omega^{-1}}}{|D(\xi, 2r)|^2} \leqslant [\omega]_{\mathcal{A}_2}.
\end{equation*}
Recall that a pseudohyperbolic disk $D(z,r)$ is a Euclidean disk with center and radius given by
\begin{equation*}
\mathfrak{C}=\frac{1-r^2}{1-r^2|z|^2}z,\quad \mathfrak{R}=\frac{1-|z|^2}{1-r^2|z|^2}r.
\end{equation*}
Combining the above with Cauchy-Schwarz's inequality gives
\begin{align*}
|D(\xi, 2r)|_\omega &\leqslant [\omega]_{\mathcal{A}_2} \frac{|D(\xi, 2 r)|^2}{|D(\xi, 2r)|_{\omega^{-1}}}\\
&\leqslant [\omega]_{\mathcal{A}_2}\frac{|D(\xi, r)|^2}{|D(\xi, r)|_{\omega^{-1}}} \cdot \frac{|D(\xi, 2 r)|^2}{|D(\xi, r)|^2}\\
&\leqslant 4[\omega]_{\mathcal{A}_2} |D(\xi, r)|_\omega  \cdot \bigg(\frac{1-r^2|\xi|^2}{1-4r^2|\xi|^2}\bigg)^2  \\
&< 8 [\omega]_{\mathcal{A}_2} |D(\xi, r)|_\omega,
\end{align*}
where the last inequality is due to $r \leqslant\frac{1}{4}$.
This completes the proof of Lemma \ref{two disks}.
\end{proof}

We now turn to the representation theory of the space $L_h^2(\omega)$.
These results and their proof strategies are motivated by Luecking's works on weighted Bergman space (\cite{Luecking1985}, \cite{Luecking AJM}).

Before studying the representation theory of $L_h^2(\omega)$, we need to recall the concept of $\epsilon$-lattice in the unit disk.
Let $\epsilon \in (0, 1)$, a sequence $\{a_n\}_{n=1}^{\infty}$ in the unit disk is  called an $\epsilon$-lattice in the pseudo-hyperbolic metric if
\begin{itemize}
  \item $\mathbb D=\bigcup_{n=1}^\infty D(a_n, \epsilon)$ and
  \item $\inf\limits_{n \neq m} \Big|\frac{a_n-a_m}{1-\overline{a_n} a_m}\Big|\geqslant \frac{\epsilon}{2}.$
\end{itemize}

Now, we are ready to state the atomic decomposition for $L^2_h(\omega)$.
\begin{thm}\label{Atomic Decomposition}
Let $\omega$ be an $\mathcal {A}_2$ weight. Then there is an $\epsilon$-lattice $\{a_n\}_{n=1}^{\infty}$ such that for each $f\in L_h^2(\omega)$ we have
$$f(z)=\sum_{n=1}^\infty c_n (1-|a_n|^2)^2 |D(a_n, \epsilon)|_\omega^{-\frac{1}{2}} R_{a_n}(z) $$
for some sequence $\{c_n\}$ in $\ell^2(\mathbb N)$, where
$$R_\lambda(z)=\frac{1}{(1-\overline{z}\lambda)^2}+\frac{1}{(1-\overline{\lambda}z)^2}-1$$ is the reproducing kernel for $L_h^2$ at $\lambda\in\dd$.
\end{thm}
\emph{Remarks.}  We have the following estimate of the module of  $R_\lambda$: there exists an $r_0 \in (0, \frac{1}{4}]$  such that if $0<r\leqslant r_0$, then
$$\frac{\frac{1}{2}}{(1-|\lambda|)^2} \leqslant |R_\lambda(z)| \leqslant \frac{3}{(1-|\lambda|)^2}$$
for all $z \in D(\lambda, r)$. For the proof of this fact, we refer to Lemma 2.2 in \cite{Choi}. In what follows,  we will use $r_0$  to denote the constant provided in this remarks.

To prove Theorem \ref{Atomic Decomposition}, we need to establish a harmonic version of Luecking's theorems in \cite{Luecking1985} and \cite{Luecking AJM}.
\begin{thm}\label{L:key00}
Let $\omega$ be an $\mathcal {A}_2$ weight. Then there exists an $\epsilon$-lattice $\{a_n\}_{n=1}^\infty$  for some  $0<\epsilon<\frac{1}{16}$
such that
 \begin{equation*}
 \sum_{n=1}^\infty |f(a_n)|^2 |D(a_n, \epsilon)|_\omega \approx \|f\|_{L^2(\omega)}^2
 \end{equation*}
 for all $f$ in $L_h^2(\omega)$.  That is, there exist two positive  constant $C_1$ and $C_2$ such that
 \begin{equation*}\label{E:doublecarleson}
 C_1\|f\|_{L^2(\omega)}^2 \leqslant \sum_{n=1}^\infty |f(a_n)|^2 |D(a_n, \epsilon)|_\omega \leqslant C_2\|f\|_{L^2(\omega)}^2.
 \end{equation*}
 for all $f$ in $L_h^2(\omega)$.
\end{thm}

Once Theorem \ref{L:key00} is established, we can quickly present a proof of Theorem \ref{Atomic Decomposition} as follows.

\begin{proof}[Proof of Theorem \ref{Atomic Decomposition}] Note that both $\omega$ and $\omega^{-1}$ are $\mathcal{A}_2$ weights, it follows from Theorem \ref{L:key00} that we can choose $\epsilon\in (0, \frac{1}{16})$ and an $\epsilon$-lattice $\{a_n\}_{n=1}^\infty$ such that
$$\|g\|_{L^2(\omega^{-1})}^2\approx  \sum_{n=1}^\infty |g(a_n)|^2 |D(a_n, \epsilon)|_{\omega^{-1}}.$$
By Cauchy-Schawrz inequality and the definition of $\mathcal{A}_2$ weight,
\begin{equation*}
|D(a_n, \epsilon)|^2 \leqslant |D(a_n, \epsilon)|_\omega \cdot |D(a_n, \epsilon)|_{\omega^{-1}}\leqslant [\omega]_{\mathcal{A}_2}|D(a_n, \epsilon)|^2.
\end{equation*}
Therefore,
$$\|g\|_{L^2(\omega^{-1})}^2\approx \sum_{n=1}^\infty |g(a_n)|^2 (1-|a_n|^2)^4 |D(a_n, \epsilon)|_\omega^{-1} .$$
This implies that the linear operator $\mathscr{L}: L_h^2(\omega^{-1})\rightarrow \ell^2(\mathbb Z)$ defined  by
$$\mathscr L (g):=\bigg\{g(a_n) (1-|a_n|^2)^2 |D(a_n, \epsilon)|_\omega^{-\frac{1}{2}}\bigg\}_{n=1}^\infty$$
is bounded below and so its range is closed. It follows from  the closed range theorem that $\mathscr L ^*$  is surjective.

From the proof of Lemma \ref{duality},  we have
$$g(a_n)=\langle g, R_{a_n}\rangle_{L^2(dA)}\ \ \ \  (*)$$
for each $g\in L_h^2(\omega^{-1})$ and every $n\geqslant 1$. Let $\{c_n\}_{n=1}^\infty \in \ell^2(\mathbb N)$. Using $(*)$ to  obtain that
$$\mathscr L ^*(\{c_n\})(z)=\sum_{n=1}^\infty c_n R_{a_n}(z) (1-|a_n|^2)^2 |D(a_n, \epsilon)|_\omega^{-\frac{1}{2}},$$
which gives the desired result. This completes the proof of Theorem \ref{Atomic Decomposition}.
\end{proof}

We  now turn to the proof of Theorem \ref{L:key00}. Let $0<\epsilon<\frac{1}{16}$ and $\{a_n\}_{n=1}^\infty\subset\mathbb{D}$ be an $\epsilon$-lattice. Define a measure $\mu=\mu_\epsilon$ on $\dd$ by
\begin{equation*}\label{E:mu}
\mu(z)=\sum_{n=1}^\infty\delta_{a_n}(z)\big|D(a_n,\frac{\epsilon}{4})\big|_\omega,
\end{equation*}
where $\delta_{a_n}$ is the Dirac measure concentrated at $a_n$.  Indeed, the conclusion of Theorem \ref{Atomic Decomposition} tells us that $\mu$ is a Carleson and reverse  Carleson  measure for $L_h^2(\omega)$. Firstly, we establish a sufficient condition for a general (positive) measure to be the $L_h^p(\omega)$ Carleson  measure, where $0<p<\infty$.
\begin{prop}\label{mu Carleson measure}
Suppose that $\nu$ is a positive Borel measure on $\mathbb D$. If there exist an $0<r\leqslant r_0$ and a constant $C>0$ independent of $z \in \mathbb D$ such that $$\nu (D(z, r)) \leqslant C |D(z, r)|_\omega$$ for all $z\in \mathbb D$, then  $\nu$ is a Carlseon measure for $L_h^p(\omega)$ ($0<p<\infty$), i.e.,
there is a  positive  constant $C_p$ such that
$$\int_{\mathbb D}|f(z)|^pd\nu(z)\leqslant C_p \|f\|_{L^p(\omega)}^p$$
for $f\in L_h^p(\omega)$. Consequently, $\mu$ is a  Carleson  measure for $L_h^2(\omega)$, i.e.,
there is an absolute constant $C>0$  such that
$$\int_{\mathbb D}|f(z)|^2d\mu(z)=\sum_{n=1}^\infty |f(a_n)|^2 \big|D(a_n, \frac{\epsilon}{4})\big|_\omega \leqslant C \|f\|_{L^2(\omega)}^2$$
for all $f$ in $L_h^2(\omega)$.
\end{prop}
\begin{proof}
Fix an $r\leqslant r_0$. By Lemma \ref{subharmonic behavior of harmonic functions}, we obtain
$$|f(z)|^{\frac{p}{2}}\leqslant \frac{C}{|D(z, r)|} \int_{D(z, r)}|f(\xi)|^{\frac{p}{2}} dA(\xi) \ \  \  \big(z\in \mathbb D\big),$$
where $0<p<\infty $ and $C=C(p, r)$. Cauchy-Schwartz inequality and $\mathcal{A}_2$ condition give us that
\begin{align*}
|f(z)|^p & \leqslant \frac{C^2}{|D(z, r)|^2} \bigg(\int_{D(z, r)}|f(\xi)|^{\frac{p}{2}}dA(\xi)\bigg)^2\\
& \leqslant  \frac{C^2}{|D(z, r)|^2}  \bigg(\int_{D(z, r)}|f(\xi)|^p\omega(\xi)dA(\xi)\bigg)\cdot \bigg(\int_{D(z,r)}\frac{1}{\omega(\xi)}dA(\xi)\bigg)\\
&\leqslant C^2 [\omega]_{\mathcal{A}_2} \frac{\int_{D(z, r)}|f(\xi)|^p\omega(\xi)dA(\xi)}{|D(z, r)|_\omega}.
\end{align*}
Integrating the above over the unit disk to obtain
\begin{align*}
\int_{\mathbb D} |f(z)|^p d\nu(z)&\leqslant C \int_{\mathbb D} |D(z, r)|_\omega^{-1}\int_{D(z,r)}|f(\xi)|^p \omega(\xi)dA(\xi)d\nu(z)\\
&=C \int_{\mathbb D} |D(z, r)|_\omega^{-1}\int_{\mathbb D}|f(\xi)|^p \chi_{D(z, r)}(\xi)\omega(\xi)dA(\xi)d\nu(z)\\
&=C \int_{\mathbb D} \int_{\mathbb D} |D(z, r)|_\omega^{-1}|f(\xi)|^p \omega(\xi) \chi_{D(\xi, r)}(z) dA(\xi)d\nu(z),
\end{align*}
here the constant $C$ depends only on $p$ and $r$.  Note that $\xi \in D(z, r)$, we have by Lemma \ref{two disks} that
$$|D(z, r)|_\omega \geqslant C |D(\xi, r)|_\omega$$
for some absolute  constant $C$.  Therefore,
\begin{align*}
\int_{\mathbb D} |f(z)|^2d\nu(z)&\leqslant C \int_{\mathbb D} \int_{\mathbb D}|D(z, r)|_\omega ^{-1}|f(\xi)|^p \omega(\xi) \chi_{D(\xi, r)}(z) dA(\xi)d\nu(z)\\
&\leqslant C \int_{\mathbb D} \int_{\mathbb D} |D(\xi, r)|_\omega^{-1}|f(\xi)|^p \omega(\xi) \chi_{D(\xi, r)}(z)d\nu(z)dA(\xi).
\end{align*}
Now using our hypothesis on $\nu$ to get
\begin{align*}
\int_{\mathbb D} |f(z)|^p d\nu(z) &\leqslant C \int_{\mathbb D} |D(\xi, r)|_\omega^{-1}\nu(D(\xi, r)) |f(\xi)|^p \omega(\xi)   dA(\xi)\\
& \leqslant C_1 \int_{\mathbb D} |f(\xi)|^p \omega(\xi)dA(\xi),
\end{align*}
the constant $C_1 >0$ comes from the assumption  and  $C_1$ is independent of $f\in L_h^p(\omega)$.

For the second conclusion of this proposition,  it sufficient for us to show the following inequality:

$$\mu(D(a, \frac{1}{4})) \leqslant C \int_{D(a, \frac{1}{4})} \omega(z)dA(z)\ \ \ \ \  \big(a\in \mathbb D\big) $$
for some  absolute  constant $C>0$.  Indeed, by the definition of $\mu$, we have
$$\mu(D(a,\frac{1}{4}))=\sum_{\rho(a_n, a)<\frac{1}{4}} \int_{D(a_n, \frac{\epsilon}{4})}\omega dA=\sum_{\rho(a_n, a)<\frac{1}{4}}\big|D(a_n, \frac{\epsilon}{4})\big|_\omega.$$
If $\rho(a, a_n)<\frac{1}{4}$, then for each $z\in D(a_n, \frac{\epsilon}{4})$ we get
$$\rho(z, a)\leqslant \rho(z, a_n)+\rho(a_n, a)<\frac{\epsilon}{4}+\frac{1}{4}<\frac{1}{2}.$$
Which shows that
$$D(a_n, \frac{\epsilon}{4}) \subset D(a, \frac{1}{2})$$
for every $n \geqslant 1$ provided $\rho(a, a_n)<\frac{1}{4}$.

Since $D(a_n, \frac{\epsilon}{4}) \cap D(a_m, \frac{\epsilon}{4})=\varnothing$ for $n \neq m$, we obtain
$$\bigcup_{\rho(a_n, a)<\frac{1}{4}} D(a_n, \frac{\epsilon}{4})\subset D(a, \frac{1}{2})$$
and
$$\mu(D(a, \frac{1}{4})) \leqslant \big|D(a, \frac{1}{2})\big|_\omega \leqslant C \big|D(a, \frac{1}{4})\big|_\omega$$
for every $a\in \mathbb D$, the constant $C>0$ (independent of $a$) comes from Lemma \ref{two disks}. This completes the proof of  Proposition \ref{mu Carleson measure}.
\end{proof}

In order to finish the proof of Theorem \ref{L:key00},  we need to show that there is an $\epsilon \in (0, \frac{1}{16})$ such that $\mu=\mu_\epsilon$ is a reverse $L_h^2(\omega)$ Carleson measure. More precisely, we
will prove the following proposition.
\begin{prop}\label{mu reverse Carleson measure}
There exists an  $\epsilon \in (0, \frac{1}{16})$ such that
$$\int_{\mathbb D}|f(z)|^2 d\mu(z)=\sum_{n=1}^\infty |f(a_n)|^2 |D(a_n, \epsilon)|_\omega \geqslant  C \|f\|_{L^2(\omega)}^2$$
for all $f$ in $L_h^2(\omega)$, where $C>0$ is an absolute constant.
\end{prop}

The rest of this section is devoted to the proof of the above \emph{reverse Carleson inequality}.  To do so, we
need to prove the following two lemmas related to harmonic functions, which extend
the results in \cite{Luecking1985} and \cite{Luecking AJM} for weighted Bergman spaces to the present situation.
\begin{lemma}\label{gradient estimation}
Let $f$ be a harmonic function on $\mathbb D$  and $\epsilon \in (0, \frac{1}{16})$.   Then there exists a constant $C_1>0$ (independent of $z, \epsilon$ and $f$) such that
$$|f(z)-f(0)|\leqslant C_1 \epsilon \int_{D(0, \frac{1}{4})} |f(\lambda)| dA(\lambda)$$
when $|z|<\epsilon.$ As a consequence, there exists a constant $C_2>0$ (independent of $z, \epsilon$ and $f$) such that
$$|f(w)-f(\xi)|^2 \leqslant  \frac{C_2 \epsilon^2}{|D(\xi, \frac{1}{4})|_\omega}\int_{D(\xi, \frac{1}{4})} |f(\lambda)|^2 \omega(\lambda) dA(\lambda)$$
when $\xi\in D(w, \epsilon)$.
\end{lemma}
\begin{proof}
 Observe that
$$f(z)-f(0)=\int_{0}^1 \frac{\partial}{\partial t}\big[f(tz)\big]dt=\int_0^1 \big[\nabla f(tz) \cdot z \big] dt$$
for $|z|<\epsilon<\frac{1}{16}$.  Thus we have
$$|f(z)-f(0)|\leqslant \sup_{|\xi|\leqslant \epsilon} |\nabla f(\xi)|\cdot |z|.$$

Recall that
\begin{align*}
|\nabla f|^2&=2\big(|\partial f|^2+|\overline{\partial} f|^2\big)=2\big(|\partial f|^2+|\partial \overline{f}|^2\big)\\
&\leqslant \big[\sqrt{2} (|\partial f|+|\partial \overline{f}|)\big]^2,
\end{align*}
where $\overline{\partial} f= \frac{\partial f}{\partial \overline{z}}.$
By Lemma \ref{subharmonic behavior of harmonic functions}, there is an absolute constant $C>0$ such that
\begin{align*}
|\nabla f(\xi)| & \leqslant \frac{2\sqrt{2}C}{(1-\frac{1}{16})^3}\cdot \frac{1}{|D(\xi, \frac{1}{16})|^{\frac{3}{2}}} \int_{D(\xi, \frac{1}{16})} |f(\lambda)| dA(\lambda)\\
&=\frac{2\sqrt{2}C}{\big(1-\frac{1}{16}\big)^3}\cdot \bigg(\frac{1-(\frac{1}{16})^2|\xi|^2}{\frac{1}{16}(1-|\xi|^2)}\bigg)^3\int_{D(\xi, \frac{1}{16})} |f(\lambda)| dA(\lambda).
\end{align*}
Note that $|\xi| \leqslant \epsilon<\frac{1}{16}$ and if $\lambda \in D(\xi, \frac{1}{16})$, then
$$|\lambda|<|\xi|+\frac{1}{16}|1-\overline{\lambda}\xi|<\frac{1}{4},$$
so we have $D(\xi, \frac{1}{16}) \subset D(0, \frac{1}{4})$ and
$$|\nabla f(\xi)| \leqslant C_1 \int_{D(0, \frac{1}{4})} |f(\lambda)| dA(\lambda)$$
for all $|\xi|\leqslant \epsilon$,
where the constant $C_1$  is independent of $z, \xi$ and $\epsilon$.
Therefore,
$$|f(z)-f(0)| \leqslant \sup_{|\xi|\leqslant \epsilon} |\nabla f(\xi)| \cdot |z|\leqslant C_1 |z| \int_{D(0, \frac{1}{4})} |f(\lambda)| dA(\lambda)$$
for  $|z|<\epsilon$ with $\epsilon \in (0, \frac{1}{16})$, this proves the first part of the lemma.

Let $\varphi_\xi$ be the M\"{o}bius map, then $f\circ \varphi_\xi$ is harmonic on $\mathbb D$. Changing of variable gives that
\begin{align*}
\big|f(\varphi_\xi(z))-f(\xi)\big| \leqslant \frac{ C_3 \epsilon}{|D(\xi, \frac{1}{4})|}\int_{D(\xi, \frac{1}{4})} |f(\lambda)| dA(\lambda)
\end{align*}
for some absolute constant $C_3>0$.  By Cauchy-Schwartz inequality  we obtain
\begin{align*}
\big|f(\varphi_\xi(z))-f(\xi)|^2 &\leqslant \frac{ C_3^2 \epsilon^2}{|D(\xi, \frac{1}{4})|^2}
\bigg(\int_{D(\xi, \frac{1}{4})} |f(\lambda)|^2\omega(\lambda) dA(\lambda)\bigg)\cdot \bigg(\int_{D(\xi, \frac{1}{4})} \frac{1}{\omega(\lambda)}dA(\lambda)\bigg)\\
&\leqslant \frac{C_3^2 [\omega]_{\mathcal{A}_2} \epsilon^2}{\big|D(\xi, \frac{1}{4})\big|_\omega}\int_{D(\xi, \frac{1}{4})} |f(\lambda)|^2 \omega(\lambda) dA(\lambda).
\end{align*}
 Let $w=\varphi_\xi(z)$, then $|\varphi_\xi(w)|=|z|<\epsilon$, which gives that
$$|f(w)-f(\xi)|^2 \leqslant \frac{C_2\epsilon^2}{\big|D(\xi, \frac{1}{4})\big|_\omega}\int_{D(\xi, \frac{1}{4})} |f(\lambda)|^2 \omega(\lambda) dA(\lambda)$$
if $\xi \in D(w, \epsilon)$, as desired.
\end{proof}

\begin{lemma}\label{double intrgral estimation}
Let $f$ be a harmonic function and $\epsilon \in (0, \frac{1}{16})$.  Let $\mu$ be the measure defined  in Section 2. Then there exists a constant $C>0$ (independent of $\epsilon$) such that
$$\int_{\mathbb D} \int_{\mathbb D} \chi_{\epsilon} (z, \xi) |D(\xi, \epsilon)|_\omega^{-1} |f(z)-f(\xi)|^2 \omega(z) dA(z)d\mu(\xi)
\leqslant C\epsilon^2 \|f\|_{L^2(\omega)}^2$$
for all $f\in L_h^2(\omega)$, where
$$\chi_{\epsilon}(z, \xi)=
\begin{cases}
1,& \mathrm{if}\  z \in D(\xi, \epsilon),\\
0,&  \mathrm{otherwise}.
\end{cases}$$
\end{lemma}
\begin{proof} By Lemma \ref{gradient estimation}, we have
$$\chi_{\epsilon}(z, \xi)|f(z)-f(\xi)|^2 \frac{\omega(z)}{|D(\xi, \epsilon)|_\omega}
\leqslant \bigg(\frac{C \epsilon^2} {\big|{D(\xi, \frac{1}{4})}\big|_\omega } \int_{D(\xi, \frac{1}{4})} |f(\lambda)|^2 \omega(\lambda)dA(\lambda) \bigg)
\frac{\chi_{\epsilon}(z, \xi)\omega(z)}{|D(\xi, \epsilon)|_\omega}.$$
Integrating over $z\in \mathbb D$ on both sides gives
$$\int_{\mathbb D}\chi_{\epsilon}(z, \xi)|f(z)-f(\xi)|^2 \frac{\omega(z)}{|D(\xi, \epsilon)|_\omega}dA(z)\leqslant \frac{C \epsilon^2} {\big|D(\xi, \frac{1}{4})\big|_\omega} \int_{D(\xi, \frac{1}{4})} |f(\lambda)|^2 \omega(\lambda)dA(\lambda).$$
Now integrate with respect to $d\mu(\xi)$ to obtain
$$\int_{\mathbb D}\int_{\mathbb D}|f(z)-f(\xi)|^2 \frac{\chi_{\epsilon}(z, \xi) \omega(z)}{|D(\xi, \epsilon)|_\omega}dA(z)d\mu(\xi)\leqslant \int_{\mathbb D}\frac{C \epsilon^2} {\big|D(\xi, \frac{1}{4})\big|_\omega} \int_{D(\xi, \frac{1}{4})} |f(\lambda)|^2 \omega(\lambda)dA(\lambda)d\mu(\xi).$$
Using Fubini's Theorem on the right side and noting $\chi_{D(\xi, \frac{1}{4})}(\lambda)=\chi_{D(\lambda, \frac{1}{4})}(\xi)$, we have
\begin{align*}
\int_{\mathbb D}\frac{C \epsilon^2} {|D(\xi, \frac{1}{4})|_\omega} \int_{D(\xi, \frac{1}{4})} |f(\lambda)|^2 \omega(\lambda)dA(\lambda)d\mu(\xi)
=C\epsilon^2 \int_{\mathbb D} \bigg(\int_{\mathbb D} \frac{\chi_{D(\lambda, \frac{1}{4})}(\xi)}{|D(\xi, \frac{1}{4})|_\omega}d\mu(\xi)\bigg)
|f(\lambda)|^2\omega(\lambda)dA(\lambda)
\end{align*}

Since $\lambda \in D(\xi, \frac{1}{4})$, Lemma \ref{two disks} tells us  that there is an absolute constant $C>0$ such that
$$\big|D(\xi, \frac{1}{4})\big|_\omega \geqslant C  \big|D(\lambda, \frac{1}{4})\big|_\omega.  $$
Thus we obatin
$$\int_{\mathbb D} \frac{\chi_{D(\lambda, \frac{1}{4})}(\xi)}{|D(\xi, \frac{1}{4})|_\omega}d\mu(\xi) \leqslant C^{-1} \frac{\mu(D(\lambda, \frac{1}{4}))}
{|D(\lambda, \frac{1}{4})|_\omega}.$$
By Lemma \ref{mu Carleson measure}, we have
$$\mu(D(\lambda, \frac{1}{4})) \cdot \big|D(\lambda, \frac{1}{4})\big|_\omega^{-1}\leqslant C_1$$
for some constant $C_1>0$ (independent of $\epsilon$), completing the proof.
\end{proof}

Now we are ready to prove the reverse Carleson inequality in Proposition \ref{mu reverse Carleson measure}.
\begin{proof}[Proof of Proposition \ref{mu reverse Carleson measure}] Recall that $\mu$ satisfies the condition: $$\mu(D(a, \frac{1}{4})) \leqslant C \big|D(a, \frac{1}{4})\big|_\omega $$ for all $a\in \mathbb D$. Applying Lemma  \ref{double intrgral estimation} to  $\epsilon \in (0, \frac{1}{16})$ we have
$$\bigg[\int_{\mathbb D} \int_{\mathbb D}  \frac{\chi_{\epsilon} (z, \xi)} {|D(\xi, \epsilon)|_\omega} |f(z)-f(\xi)|^2 \omega(z) dA(z)d\mu(\xi)
\bigg]^{\frac{1}{2}}\leqslant C \epsilon \|f\|_{L^2(\omega)}.$$
The triangle inequality  gives
$$I-J \leqslant C \epsilon \|f\|_{L^2(\omega)},$$
where
 $$I:=\bigg[\int_{\mathbb D} \int_{\mathbb D}  \frac{\chi_{\epsilon} (z, \xi)} {|D(\xi, \epsilon)|_\omega} |f(z)|^2 \omega(z) dA(z)d\mu(\xi)
\bigg]^{\frac{1}{2}}$$
and
$$J:=\bigg[\int_{\mathbb D} \int_{\mathbb D}  \frac{\chi_{\epsilon} (z, \xi)} {|D(\xi, \epsilon)|_\omega} |f(\xi)|^2 \omega(z) dA(z)d\mu(\xi)\bigg]^{\frac{1}{2}}.$$

For the first integral $I$, we note that for each $z\in \mathbb D$,
\begin{align*}
\int_{\mathbb D} \frac{\chi_{\epsilon} (z, \xi)} {|D(\xi, \epsilon)|_\omega } d\mu(\xi)&=\int_{D(z, \epsilon)} \frac{d\mu(\xi)} {|D(\xi, \epsilon)|_\omega } \geqslant C_1 \frac{\mu(D(z, \epsilon))}{|D(z, \epsilon)|_\omega },
\end{align*}
where the $``\geqslant"$  follows from Lemma \ref{two disks} and $C_1$ is an absolute constant.

Since $\mathbb D=\bigcup_{n=1}^\infty D(a_n, \epsilon)$,  we can select a disk $D(a_j, \epsilon)$ such that $z\in D(a_j, \epsilon)$ for each $z\in \mathbb D$. Applying  Lemma \ref{two disks} to get
\begin{align*}
\mu(D(z, \epsilon))&=\sum_{n=1}^\infty \delta_{a_n}(D(z, \epsilon))\int_{D(a_n, \frac{\epsilon}{4})}\omega dA\\
&\geqslant \big|D(a_j, \frac{\epsilon}{4})\big|_\omega \geqslant C_2 |D(a_j, \epsilon)|_\omega
\geqslant C_3 |D(z, \epsilon)|_\omega,
\end{align*}
 where $C_2 $  and $C_3$ are  absolute positive constants. Therefore, we have
$$\int_{\mathbb D} \frac{\chi_{\epsilon} (z, \xi)} {|D(\xi, \epsilon)|_\omega} d\mu(\xi)\geqslant \widetilde{C}$$
for some absolute constant $\widetilde{C}>0$. These give us  that $$I\geqslant \widetilde{C} \|f\|_{L^2(\omega)}.$$

For the second integral, we observe that
$$\int_{\mathbb D}\chi_{\epsilon}(z, \xi) \omega(z)dA(z)=|D(\xi, \epsilon)|_\omega,$$
to get  $$J=\bigg(\int_{\mathbb D} |f(\xi)|^2 d\mu(\xi)\bigg)^{\frac{1}{2}}.$$
Thus we have
$$\widetilde{C} \|f\|_{L^2(\omega)}-\bigg(\int_{\mathbb D} |f(\xi)|^2d\mu(\xi)\bigg)^{\frac{1}{2}} \leqslant I-J \leqslant C \epsilon \|f\|_{L^2(\omega)}.$$
Equivalently,
$$(\widetilde{C}- C\epsilon)\|f\|_{L^2(\omega)}\leqslant \bigg( \int_{\mathbb D} |f(\xi)|^2d\mu(\xi)\bigg)^{\frac{1}{2}}$$
for each $0<\epsilon<\frac{1}{16}$.

Since $C, \widetilde{C}$  are both independent of $\epsilon$, we can choose $$0<\epsilon< \min \bigg\{\frac{1}{16}, \frac{\widetilde{C}}{2C}, r_0\bigg\}$$ such that
\begin{align*}
\|f\|_{L^2(\omega)}^2 \leqslant \frac{1}{(\widetilde{C}- C\epsilon)^2} \int_{\mathbb D} |f(\xi)|^2d\mu(\xi).
\end{align*}
Recall the definition of $\mu$, we conclude that
\begin{align*}
\|f\|_{L^2(\omega)}^2 & \leqslant \frac{4}{\widetilde{C}^2} \sum_{n=1}^\infty |f(a_n)|^2 |D(a_n, \frac{\epsilon}{4})|_\omega\\
& \leqslant \frac{4}{\widetilde{C}^2} \sum_{n=1}^\infty |f(a_n)|^2 |D(a_n, \epsilon)|_\omega
\end{align*}
This completes the proof.
\end{proof}

The proof of  Theorem  \ref{Atomic Decomposition} implies the following result immediately.
\begin{thm}\label{Atomic Decomposition cor}
Suppose that $\omega$ satisfies the $\mathcal{A}_2$ condition. Then there is an $\epsilon$-lattice $\{a_n\}_{n=1}^{\infty} \subset \mathbb D$ such that any $f\in L_h^2(\omega)$ has the following form:
$$f(z)=\sum_{n=1}^\infty c_n K^{\omega}_{a_n}(z) |D(a_n, \epsilon)|_\omega^{\frac{1}{2}}$$
for some sequence $\{c_n\}$ in $\ell^2(\mathbb N)$, where $K^{\omega}_{a_n}$ is the reproducing kernel for $L_h^2(\omega)$.
\end{thm}
\begin{proof}  We consider the linear map $\mathscr S: L_h^2(\omega)\rightarrow \ell^2(\mathbb N)$:
$$\mathscr S f=\bigg\{f(a_n) |D(a_n, \epsilon)|_\omega^{\frac{1}{2}}\bigg\}_{n=1}^\infty.$$
 Since $L_h^2(\omega)$ is a Hilbert space, we apply  Propositions \ref{mu Carleson measure} and \ref{mu reverse Carleson measure}
 to deduce that
 $\mathscr S^*: \ell^2(\mathbb N)\rightarrow L_h^2(\omega)$ is surjective and
 $$\big\langle \mathscr S^* (\{c_n\}), \  f \big\rangle_{L^2(\omega)}=\bigg\langle \sum_{n=1}^\infty c_n K^{\omega}_{a_n}(z) |D(a_n, \epsilon)|_\omega^{\frac{1}{2}}, \ f \bigg\rangle_{L^2(\omega)}$$
 for $\{c_n\}\in \ell^2(\mathbb N)$ and $f\in L^2_h(\omega)$. Therefore,
 $$\mathscr S^*(\{c_n\})(z)=\sum_{n=1}^\infty c_n K^{\omega}_{a_n}(z) |D(a_n, \epsilon)|_\omega^{\frac{1}{2}}.$$
 This completes the proof of this theorem.
 \end{proof}

\section{Boundedness and Compactness of $T_\nu$ on $L^2_h(\omega)$}

 Recall Toeplitz operator $T_\nu$ initially defined on a dense subspace of $L^2_h(\omega)$ is given by
 $$T_\nu f(z)=\int_{\mathbb D}f(\lambda)\overline{K^{\omega}_z(\lambda)}d\nu(\lambda) \ \ \  \big(z\in \mathbb D\big).$$
In this section, we will characterize the boundedness and compactness of $T_\nu$ on $L_h^2(\omega)$ via Berezin transform and Carleson measure for $L_h^2(\omega)$. Firstly, we define the Berezin transform $\widetilde{\nu}$ of $\nu$ as follows:
$$\widetilde{\nu}(z)=\frac{1}{\|R_z\|_{L^2(\omega)}^2}\int_{\mathbb D}|R_z(\lambda)|^2d\nu(\lambda), $$
where $$R_z(\lambda)=\frac{1}{(1-\overline{\lambda}z)^2}+\frac{1}{(1-\overline{z}\lambda)^2}-1$$ is the reproducing kernel for  $L_h^2$. The first main result of this section is Theorem \ref{Boundedness1}.

\begin{thm}\label{Boundedness1}
Let $\nu$ be a  positive finite Borel measure  on $\mathbb D$ and $\omega$  be an  $\mathcal{A}_2$ weight. The following  conditions are equivalent:\\
(1) $T_\nu$ extends to a bounded linear operator on $L_h^2(\omega)$;\\
(2) $\nu$ is a Carlseon measure for $L_h^2(\omega)$;\\
(3) There exist an $0<r\leqslant r_0$ and a  constant $C>0$ independent of $z \in \mathbb D$ such that $$\nu(D(z, r)) \leqslant C |D(z, r)|_\omega$$ for all $z\in \mathbb D$;\\
(4) Berezin transform $\widetilde{\nu}$  is bounded.
\end{thm}

To prove Theorem \ref{Boundedness1}, we need the following useful lemma.
\begin{lemma}\label{norm of kernel}
Let $\omega \in \mathcal{A}_2$. If  $0<r\leqslant r_0$, then there is a constant $C=C(r)$ such that
$$\frac{|D(\lambda, r)|_\omega}{2(1-|\lambda|)^4} \leqslant \|R_\lambda\|_{L^2(\omega)}^2 \leqslant C \frac{|D(\lambda, r)|_\omega}{(1-|\lambda|)^4}$$
for all $\lambda \in \mathbb D$.
\end{lemma}
\begin{proof} Let $\lambda\in\dd$. By the remarks below Theorem \ref{Atomic Decomposition},  there exists an $r_0 \in (0, \frac{1}{4}]$  such that if $0<r\leqslant r_0$, then
$$\frac{\frac{1}{2}}{(1-|\lambda|)^2} \leqslant |R_\lambda(z)| \leqslant \frac{3}{(1-|\lambda|)^2}$$
for all $z \in D(\lambda, r)$.
It follows that for each $r \in (0, r_0]$ we have
\begin{align*}
\|R_\lambda\|_{L^2(\omega)}^2 &=\int_{\mathbb D} |R_\lambda(z)|^2 \omega(z)dA(z)\\
&\geqslant \int_{D(\lambda, r)} |R_\lambda(z)|^2 \omega(z)dA(z)\\
&\geqslant \frac{1}{4} \int_{D(\lambda, r)} \frac{\omega(z)}{(1-|\lambda|)^4} dA(z)\\
&=\frac{|D(\lambda, r)|_\omega}{4(1-|\lambda|)^4}.
\end{align*}
To show the other inequality, observe that
$$|z-\lambda|<r(1-|\lambda|)<r|1-\overline{z}\lambda| \  \   \ \  \big(z, \lambda \in \mathbb D \big),$$ we have $$S(\lambda, r):=\{z\in \mathbb D: |z-\lambda|<r(1-|\lambda|)\}\subset D(\lambda, r).$$
Thus, we have  by Lemma 2.1 in \cite{Con} that
$$ C_1\frac{|S(\lambda, r)|_\omega}{(1-|\lambda|)^4} \leqslant \|K_\lambda\|_{L^2(\omega)}^2 \leqslant C_2 \frac{|S(\lambda, r)|_\omega}{(1-|\lambda|)^4}\leqslant  C_2 \frac{|D(\lambda, r)|_\omega}{(1-|\lambda|)^4}$$
for some positive constants $C_1=C_1(r)$ and $C_2=C_2(r)$, here $K_\lambda(z)=\frac{1}{(1-\overline{\lambda}z)^2}$ is the reproducing kernel for $L_a^2$ at $\lambda$.
 Recall that
$$R_\lambda(z)=2\mathrm{Re} (K_\lambda(z))-1,
$$
we have
\begin{align*}
\|R_\lambda\|_{L^2(\omega)}^2 \leqslant \frac{4C|D(\lambda, r)|\omega}{(1-|\lambda|)^4} + 2\|\omega\|_{L^1}.
\end{align*}
Consequently, to complete the proof we need only to show that there is a constant $C_3$ depending only on  $r$ such that
$$\|\omega\|_{L^1} \leqslant C_3 \frac{ |D(\lambda, r)|_\omega}{(1-|\lambda|)^4} $$
for every $\lambda \in \mathbb D.$ Indeed, we may assume that $\|\omega\|_{L^1}=1$. Then it is easy to see that
$$C (1-|\lambda|)^2\leqslant  |D(\lambda, r)|$$
for some constant $C=C(r)$.
Thus we have
\begin{align*}
C (1-|\lambda|)^2 &\leqslant |D(\lambda, r)|=\int_{D(\lambda, r)} \omega(z)^{\frac{1}{2}} \omega(z)^{-\frac{1}{2}}dA(z)\\
&\leqslant |D(\lambda, r)|_\omega^{\frac{1}{2}}\cdot |D(\lambda, r)|_{\omega^{-1}}^{\frac{1}{2}}\\
& \leqslant |D(\lambda, r)|_\omega^{\frac{1}{2}}\cdot \|\omega^{-1}\|_{L^1}^{\frac{1}{2}}.
\end{align*}
This shows that
$$\frac{|D(\lambda, r)|_\omega}{(1-|\lambda|)^4}\geqslant C_3,$$
to complete the proof of Lemma \ref{norm of kernel}.
\end{proof}

Now we are ready to prove Theorem \ref{Boundedness1}.
\begin{proof}[Proof of Theorem \ref{Boundedness1}] Our strategy is $(4)\Rightarrow(3)\Rightarrow(2)\Rightarrow(1)\Rightarrow(4)$.

$(4)\Rightarrow (3)$: Note that there exists an $r \in (0, r_0]$  such that
$$\frac{\frac{1}{2}}{(1-|z|)^2} \leqslant |R_z(\lambda)| \leqslant \frac{3}{(1-|z|)^2}$$
for $\lambda\in D(z, r)$.  From the definition of $\widetilde{\nu}$, there is a constant $C>0$ such that
 $$\frac{1}{\|R_z\|_{L^2(\omega)}^2}\int_{D(z,r)}|R_z(\lambda)|^2d\nu(\lambda)\leqslant \widetilde{\nu}(z)\leqslant C$$
 for all $z\in \mathbb D.$  Combining these with  Lemma \ref{norm of kernel} gives us that
 $$\nu(D(z, r))\leqslant C|D(z,r)|_\omega\ \ \  \big(z\in \mathbb D \big)$$
 for some positive constant $C=C(r)$.

$(3)\Rightarrow (2)$: This was proved in Proposition \ref{mu Carleson measure}.

$(2)\Rightarrow (1)$: Assume that $\nu$ is a Carleson measure. By Condition $(2)$ and Lemma \ref{norm of kernel}, we obtain Condition $(3)$. Now
Proposition \ref{mu Carleson measure} implies that $\nu$ is also a Carleson measure for $L_h^1(\omega)$.  Then for  $f, g$ are harmonic on a neighborhood of $\overline{\mathbb D}$ (by Theorem \ref{Atomic Decomposition}, these functions are dense in $L_h^2(\omega)$),  we have
$$\big|\langle T_\nu f, g\rangle \big|\leqslant \int_{\mathbb D} |f(z)g(z)|d\nu(z)\leqslant C \|fg\|_{L^1(\omega)}\leqslant C\|f\|_{L^2(\omega)}\|g\|_{L^2(\omega)},$$ which  means that $T_\nu$ is bounded.

$(1)\Rightarrow (4)$: Suppose that $T_\nu$ is bounded on $L_h^2(\omega)$. We consider the partial sum $\sigma_N=\sum_{n=1}^N t_n K^\omega_{a_n}$, where $N\geqslant 1$, $\{t_n\}$ are complex numbers  and $\{a_n\}\subset \mathbb D$.  Direct calculation shows that
$$\|\sigma_N\|_{L^2(\nu)}\leqslant C \|\sigma_N\|_{L^2(\omega)}$$
for some constant $C>0$.  This implies that if $\lim\limits_{N\rightarrow \infty}\|\sigma_N-g\|_{L^2(\omega)}=0$ for some $g\in L_h^2(\omega)$, then
$$\lim_{N \rightarrow \infty} \langle f, \sigma_N \rangle_{L^2(\nu)}=\langle f, g\rangle_{L^2(\nu)}$$
for every $f\in L^2_h(\omega)$. Applying Theorem \ref{Atomic Decomposition cor} to obtain
$$\langle T_\nu f, g\rangle_{L^2(\omega)}=\langle f, g\rangle_{L^2(\nu)}$$
for every $f, g \in L_h^2(\omega)$. In particular, we have
$$\langle T_\nu R_z, R_z \rangle_{L^2(\omega)}=\langle R_z,   R_z\rangle_{L^2(\nu)}=\widetilde{\nu}(z)\|R_z\|_{L^2(\omega)}^2,$$
to get $\widetilde{\nu}(z)\leqslant \|T_\nu\|$ for all $z\in \mathbb D$. The proof of Theorem \ref{Boundedness1} is complete now.
\end{proof}

From the above theorem, it is natural to characterize the compactness of $T_\nu$ via  vanishing Carleson measure. In fact, we will characterize the compact Toeplitz operators with positive measures as the symbols  via not only  vanishing Carlson measure (for the $\mathcal{A}_2$ weighted harmonic Bergman space) but also  Berezin transform.
\begin{thm}\label{Comapctness1}
Let $\nu$ be a  positive finite Borel measure  on $\mathbb D$ and $\omega\in \mathcal{A}_2$. The following  conditions are equivalent:\\
(1) $T_\nu$ is compact on $L_h^2(\omega)$;\\
(2) $\nu$ is a vanishing Carlseon measure for $L_h^2(\omega)$, i.e.,
$$\lim_{|z|\rightarrow 1^-} \frac{\nu(D(z, r))}{|D(z, r)|_\omega}=0$$
for some $r\in (0, 1)$;\\
(3) The Berezin transform $\lim\limits_{|z|\rightarrow 1^-}\widetilde{\nu}(z)=0.$
\end{thm}
\begin{proof}
We will show that $(2)\Rightarrow (1)\Rightarrow (3)\Rightarrow (2)$.

$(2)\Rightarrow (1):$  To prove (1), we need only to show that the inclusion operator $i: L_h^2(\omega)\rightarrow L^2(\nu)$ is compact, i.e.,
$$\lim_{n\rightarrow \infty}\int_{\mathbb D} |f_n(z)|^2 d\nu(z)=0$$
whenever $\|f_n\|_{L^2(\omega)}\rightarrow 0 \ (n\rightarrow \infty)$, where $\{f_n\}_{n=1}^\infty$ is a bounded sequence in $L_h^2(\omega)$ which
converges to zero uniformly on each compact subset of $\mathbb D$.

 From the proof of Proposition \ref{mu Carleson measure}, there exists a positive constant $C=C(r)$ such that
\begin{align*}
\int_{\mathbb D} |f_n(z)|^2 d\nu(z) &\leqslant C \int_{\mathbb D}|D(\xi, r)|_\omega ^{-1}\nu(D(\xi, r)) |f_n(\xi)|^2 \omega(\xi)   dA(\xi)\\
&=C \bigg(\int_{|\xi|\leqslant s }+\int_{|\xi|>s } \bigg)\frac{\nu(D(\xi, r))}{ |D(\xi, r)|_\omega} |f_n(\xi)|^2 \omega(\xi) dA(\xi),
\end{align*}
where $s\in (0, 1)$. Under the assumption in (2), we can choose an $s_0\in (0, 1)$ to make the second  integral as small as we like; Fix $s_0$,
it is easy to show the first integral converges to zero,  since $f_n\rightarrow 0 \ (n\rightarrow \infty)$ uniformly on compact subsets.  This proves $(2)\Rightarrow (1)$.

$(1)\Rightarrow (3):$  Observe that
$$\widetilde{\nu}(z)=\bigg\langle T_\nu \frac{R_z}{\|R_z\|_{L^2(\omega)}}, \ \frac{R_z}{\|R_z\|_{L^2(\omega)}}\bigg\rangle_{L^2(\omega)} \leqslant \bigg\|T_\nu  \bigg(\frac{R_z}{\|R_z\|_{L^2(\omega)}}\bigg)\bigg\|_{L^2(\omega)}.$$
So, it sufficient for us to show that $\frac{R_z}{\|R_z\|_{L^2(\omega)}}$ converges to zero weakly in $L_h^2(\omega)$ as $|z|\rightarrow 1^-.$
Note  that $\frac{R_z}{\|R_z\|_{L^2(\omega)}}$ is an unit vector in $L_h^2(\omega)$, we need only to show it converges to zero uniformly on compact subsets of $\mathbb D$ as $|z|\rightarrow 1^-$.
Observe that Lemma \ref{norm of kernel} implies that there exists a positive  constants $C=C(r_0)$ such that
$$\bigg|\frac{R_z(\lambda)}{\|R_z\|_{L^2(\omega)}}\bigg|^2
\leqslant   C |R_z(\lambda)|^2 \int_{D(z, r_0)}{\omega^{-1}}dA.$$
It is clear that $\frac{R_z}{\|R_z\|_{L^2(\omega)}}$ converges to zero uniformly (as $|z|\rightarrow 1^-$) on each disk $|\lambda|\leqslant s<1$, since $|D(z, r_0)|\rightarrow 0$ as $|z|\rightarrow 1^-$ and $\omega^{-1}\in L^1(dA)$.

$(3)\Rightarrow (2):$  By the definition of  $\widetilde{\nu}$ and Lemma \ref{norm of kernel},
there exists a constant $C=C(r_0)>0$  such that
\begin{align*}
\widetilde{\nu}(z) & =\frac{1}{\|R_z\|_{L^2(\omega)}^2}\int_{\mathbb D}|R_z(\lambda)|^2d\nu(\lambda)\\
& \geqslant C \frac{(1-|z|)^4}{|D(z, r_0)|_\omega} \int_{D(z, r_0)}|R_z(\lambda)|^2d\nu(\lambda)\\
& \geqslant  \frac{C}{4} \frac{\nu(D(z, r_0))}{|D(z, r_0)|_\omega}.
\end{align*}
Then the desired result follows, to complete the proof of Theorem \ref{Comapctness1}.
\end{proof}

In the rest of this section, we will consider the special class of compact Toeplitz operators to give a characterization of  $\nu$ for $T_\nu$ to be in the Schatten class $\mathcal S^p\ (p\geqslant 1)$. The following theorem is the third main result in Section 3.

\begin{thm}\label{Schatten p class}
Let $\nu$ be a  positive finite Borel measure  on $\mathbb D$ and $\omega\in \mathcal{A}_2$. Then for $ 1\leqslant p< \infty$,
$T_\nu \in \mathcal S^p$ if and only if
$$\sum_{n= 1}^\infty\bigg(\frac{\nu(D(a_n, \epsilon))}{|D(a_n, \epsilon)|_\omega}\bigg)^p<+\infty,$$
where $\big(\{a_n\}_{n=1}^\infty,  \ \epsilon\big)$ is the $\epsilon$-lattice obtained by Theorem \ref{Atomic Decomposition}.
\end{thm}

In order to prove the above result, we need one more lemma, which is a straightforward consequence of Lemmas  \ref{subharmonic behavior of harmonic functions} and \ref{norm of kernel}.
\begin{lemma}\label{Schatten class lemma}
Let $\omega \in \mathcal{A}_2$ and $0<r\leqslant r_0$. There exists a constant $C=C(r)>0$ such that
$$C^{-1} \leqslant  K^\omega_z(z) \cdot |D(z, r)|_\omega \leqslant C$$
for $z\in \mathbb D$, where $K^\omega_z$ is the reproducing kernel of $L_h^2(\omega)$.
\end{lemma}
\begin{proof} Note that $K^{\omega}_{\lambda}(\lambda)=\|K^\omega_\lambda\|_{L^2(\omega)}^2$
for each $\lambda\in \mathbb D$.  Applying  Lemma   \ref{subharmonic behavior of harmonic functions} to the function $K^\omega_\lambda(z)$, we get a constant $C$ depends only on $r$ such that
$$|K^\omega_\lambda(z)|^2\leqslant \frac{C \| K^\omega_\lambda\|_{L^2(\omega)}^2}{|D(z, r)|_\omega}=\frac{C  K^\omega_\lambda(\lambda)}{ |D(z, r)|_\omega}.$$
Taking $\lambda=z$ to get the inequality on the right hand side in Lemma \ref{Schatten class lemma}.

For the reverse inequality, note that for each $z\in \mathbb D$  we have
\begin{align*}
\frac{1}{(1-|z|)^2}& \leqslant \frac{2}{(1-|z|^2)^2}-1\\
&=R_z (z)=\big\langle R_z, K^\omega_z\big\rangle_{L^2(\omega)}\\
&\leqslant \|R_z\|_{L^2(\omega)}\cdot \|K^\omega_z\|_{L^2(\omega)}\\
& \leqslant C \frac{|D(z, r)|_\omega^{\frac{1}{2}}}{(1-|z|)^2}\cdot \sqrt{K^\omega_z(z)},
\end{align*}
where the constant $C$ comes from Lemma \ref{norm of kernel}. This finishes the proof.
\end{proof}

We are ready to prove Theorem \ref{Schatten p class}. The method of its proof  is quite standard.
\begin{proof}[Proof of Theorem \ref{Schatten p class}] Suppose that the  series converges.  We consider the $\epsilon$-lattice $\{a_n\}_{n=1}^\infty$ given by Theorem \ref{Atomic Decomposition}, recall that $\epsilon<r_0$ (see the proof of Proposition \ref{mu reverse Carleson measure}). For an arbitrary orthonormal basis
$\{e_n\}_{n=1}^\infty$ of $L_h^2(\omega)$,
we have
\begin{align*}
\sum_{n=1}^\infty \langle T_\nu e_n, e_n \rangle &=\sum_{n=1}^\infty \int_{\mathbb D} |e_n(z)|^2 d\nu(z)=\int_{\mathbb D} K^\omega_z(z)d\nu(z)\\
& \leqslant \sum_{n=1}^\infty \int_{D(a_n, \epsilon)} K^\omega_z(z)d\nu(z)\\
& \leqslant C \sum_{n=1}^\infty \int_{D(a_n, \epsilon)} |D(z, \epsilon)|_\omega ^{-1}d\nu(z),
\end{align*}
the constant $C$ comes from Lemma \ref{Schatten class lemma}, which depends only on $\epsilon$. Note that $\rho(z, a_n)<\epsilon$ for every $n\geqslant 1$,  by Lemma \ref{two disks} and its proof we can choose a constant $C_1=C_1(\epsilon)$ such that
$$\sum_{n=1}^\infty \langle T_\nu e_n, e_n \rangle \leqslant
 C_1 \sum_{n=1}^\infty \frac{\nu(D(a_n, \epsilon))}{|D(a_n, \epsilon)|_\omega}.$$
 This shows that $T_\nu \in \mathcal S^1$ provided $\sum_{n=1}^\infty \frac{\nu(D(a_n, \epsilon))}{|D(a_n, \epsilon)|_\omega}<+\infty$.

On the other hand, if $\sup\limits_{n\geqslant 1}\frac{\nu(D(a_n, \epsilon))}{|D(a_n, \epsilon)|_\omega}<+\infty$,
then by the proof Theorem 7.4 in \cite{Zhu2007} (or the proof of $(3)\Rightarrow (2)$ in Theorem \ref{Boundedness1}), we deduce that $T_\nu$ is bounded on $L_h^2(\omega)$, i.e., $T_\nu \in \mathcal S^\infty$. Now applying the interpolation theorem for the Schatten classes (see Theorem 2.6 in \cite{Zhu2007} if needed), we  obtain that $T_\nu \in \mathcal S^p$ for each $p\in (1, +\infty)$ if
$$\sum_{n=1}^\infty \bigg( \frac{\nu(D(a_n, \epsilon))}{|D(a_n, \epsilon)|_\omega}\bigg)^p<+\infty.$$

Conversely, we assume that $T_\nu \in \mathcal S^p$ for $1\leqslant p<\infty$.  We recall  by Theorem \ref{Atomic Decomposition} that each $f\in L_h^2(\omega)$ has the following form:
\begin{align*}
f(z)&=\sum_{n=1}^\infty c_n (1-|a_n|)^2 R_{a_n}(z)|D(a_n, \epsilon)|_\omega^{-\frac{1}{2}}\\
&:=\sum_{n=1}^\infty c_n h_n(z),
\end{align*}
where $\{c_n\}_{n=1}^\infty \in \ell^2(\mathbb N)$.
 Cauchy-Schwartz inequality shows that
$$\|f\|_{L^2(\omega)}^2 \leqslant \bigg(\sum_{n=1}^\infty |c_n|^2 \bigg)\cdot \bigg(\sum_{n=1}^\infty \frac{(1-|a_n|)^4}{|D(a_n, \epsilon)|_\omega} \|R_{a_n}\|_{L^2(\omega)}^2\bigg).$$

Since $\epsilon<r_0$, we have by Lemma \ref{norm of kernel} that
$$\|f\|_{L^2(\omega)}^2 \leqslant C \sum_{n=1}^\infty |c_n|^2$$
where the constant  $C>0$ depends only on $\epsilon$.

Fix an orthonormal basis $\{e_n\}_{n=1}^\infty$ for $L_h^2(\omega)$ and define a linear operator $\mathscr A$ on $L_h^2(\omega)$ by
$$\mathscr A \bigg( \sum_{n=1}^\infty c_n e_n \bigg)=\sum_{n=1}^\infty c_n h_n.$$
Then $\mathscr A$ is a bounded surjective linear operator on $L_h^2(\omega)$. Thus the  $\mathscr A^*$ is well-defined on $L_h^2(\omega)$ and
$\mathscr A^*T_\nu \mathscr A \in \mathcal S^p$, so that
$$\sum_{n=1}^\infty \big| \langle \mathscr A^*T_\nu \mathscr A e_n, e_n \rangle_{L^2(\omega)}\big|^p<+\infty.$$
On the other hand,
\begin{align*}
\sum_{n=1}^\infty \big| \langle \mathscr A^*T_\nu \mathscr A e_n, e_n \rangle_{L^2(\omega)}\big|^p&=\sum_{n=1}^\infty \big| \langle T_\nu \mathscr A e_n,  \mathscr A e_n \rangle_{L^2(\omega)}\big|^p\\
&=\sum_{n=1}^\infty \big| \langle T_\nu h_n,  h_n\rangle_{L^2(\omega)}\big|^p \ \ \ \big(\mathrm{by \ the \  definition \ of} \ \mathscr A \big)\\
&=\sum_{n=1}^\infty \bigg( \int_{\mathbb D} |h_n(z)|^2  d\nu(z)\bigg)^p\\
&\geqslant \sum_{n=1}^\infty \bigg( \int_{D(a_n, \epsilon)} |h_n(z)|^2  d\nu(z)\bigg)^p.
\end{align*}
Recall
$$|h_n(z)|^2=\frac{(1-|a_n|)^4 |R_{a_n}(z)|^2}{|D(a_n, \epsilon)|_\omega},$$
 we have  by Lemma \ref{norm of kernel} that
$$|h_n(z)|^2\geqslant \frac{1}{4|D(a_n, \epsilon)|_\omega}$$
 if $\rho (z, a_n)<\epsilon<r_0$ for each $n\geqslant 1$. Therefore,
\begin{align*}
\sum_{n=1}^\infty \big| \langle \mathscr A^*T_\nu \mathscr A e_n, e_n \rangle_{L^2(\omega)}\big|^p \geqslant 4^{-p} \sum_{n=1}^\infty
\bigg( \frac{\nu(D(a_n, \epsilon))}{|D(a_n, \epsilon)|_\omega}\bigg)^p.
\end{align*}
This completes the proof of Theorem \ref{Schatten p class}.
\end{proof}

\section{Invertibility of Toeplitz operators on $L_h^2(\omega_\alpha)$}

 A fundamental and interesting  problem is to determine when a Toeplitz operator is invertible on the Hardy or Bergman
space.  In this section, we study the invertibility problem of Toeplitz operators on the standard weighted harmonic Bergman space $L_h^2(\omega_\alpha)$ with
 $\omega_\alpha=(1+\alpha)(1-|z|^2)^\alpha, \alpha>-1$. Recall that the reproducing kernel for $L_h^2(\omega_\alpha)$ is given by
$$R^\alpha_z(\lambda)=K^\alpha_z(\lambda)+\overline{K^\alpha_z(\lambda)}-1 \ \ \  \big(z, \lambda \in \mathbb D\big),$$
where
 $$K^\alpha_z(\lambda)=\frac{1}{(1-\overline{z}\lambda)^{2+\alpha}}$$
is the reproducing kernel for the weighted Bergman space $L_a^2(\omega_\alpha)$, see \cite{Shu} if needed.

For the (unweighted) Bergman space setting, the second author and Zheng provided a necessary and sufficient condition
 for the Toeplitz operators with nonnegative symbols to be invertible  on $L_a^2$ (\cite{Zhao}). The main tool used in \cite{Zhao} is Luecking's result  on the reverse Carleson measure for Bergman space (\cite{Lue}),  which also holds for the harmonic Bergman space. More precisely, Luecking established the following result.

\begin{lemma}(\cite{Luecking1983})\label{reverse Carleson measure1}
Suppose that $G$ is a measurable subset of $\mathbb D$. Then the following are equivalent:\\
(i) There exists a  $\delta\in (0, 1)$ such that
$$ |G\cap K| \geqslant \delta |\mathbb D \cap K|$$
for every ball $K$ whose center lies on $\partial \mathbb D$;\\
(ii) $\chi_{G}dA$ is a reverse Carleson measure for $L_h^2(\omega_\alpha)$. That is,  there is a constant $C>0$ such that
$$\int_{\mathbb D}|f(z)|^2 \omega_\alpha(z) dA(z)\leqslant C \int_{G}|f(z)|^2 \omega_\alpha (z) dA(z)$$
for all $f\in L_h^2(\omega_\alpha)$.
\end{lemma}

Motivated by the ideas  and techniques  used in \cite{Zhao}, we are able to characterize the invertiblility of Toeplitz operator $T_\varphi$ ($\varphi\geqslant 0$) on $L_h^2(\omega_\alpha)$  in terms of the reverse Carleson measure and Berezin transform.

\begin{thm}\label{invertibility1} Let $\varphi\geqslant 0$ be in $L^\infty(\omega_\alpha)$. The following conditions are equivalent:\\
$(1)$ The Toeplitz operator $T_\varphi$ is invertible on $L_h^2(\omega_\alpha)$;\\
$(2)$ The Berezin transform $\widetilde{\varphi}$ is invertible in $L^\infty(\omega_\alpha)$,  where
$$\widetilde{\varphi}(z):=\frac{1}{\|R^\alpha_z\|^2_{L^2(\omega_\alpha)}}\int_{\mathbb D}\varphi(\lambda)|R_\alpha(z, \lambda)|^2 \omega_\alpha(\lambda) dA(\lambda)$$
and
$$\|R^\alpha_z\|^2_{L^2(\omega_\alpha)}=R^\alpha_z(z)=\frac{2}{(1-|z|^2)^{2+\alpha}}-1;$$\\
$(3)$ There exists $r>0$ such that  $$G:=\{z \in \mathbb D: \varphi(z)>r \},$$ $\chi_{G}dA$ is a reverse Carleson measure for $L_h^2( \omega_\alpha)$;\\
$(4)$ There exists a constant $C>0$ such that
$$\int_ {\mathbb D}|f(z)|^2 \varphi(z)\omega_\alpha(z) dA(z)\geqslant C \int_{\mathbb D}|f(z)|^2\omega_\alpha(z) dA(z)$$
for $f\in L_h^2(\omega_\alpha)$.
\end{thm}
Before giving the proof the main theorem of this section, we need another lemma, which was proved in \cite{Luecking1983}, \cite{Miao1998} and \cite{Sledd}.
\begin{lemma}\label{D-K}
Suppose that  the ball $K$ has radius $0<t<1$  and center $u=(1, 0)\in \mathbb R^2$.
Let $f$ be the harmonic  function
 $$f(\lambda)=f_s(\lambda):=\sqrt{1+\alpha}R^\alpha_{z_0}(\lambda)(1-|z_0|^2)^{\frac{2+\alpha}{2}},$$
where $z_0=(1-st)u$, $0<s<1$.
Then for each $\epsilon>0$, there exist  $s=s(\epsilon)$  and a positive constant $C=C(\epsilon)$ (independent of $K$) such that\\
$$\int_{B\setminus K} |f(\lambda)|^2 (1-|\lambda|)^\alpha dA(\lambda)<\epsilon$$
and
$$\int_{G\cap K}|f(\lambda)|^2 (1-|\lambda|)^\alpha dA(\lambda)\leqslant C \bigg(\frac{|G\cap K|}{|\mathbb D \cap K|}\bigg)^\beta,$$
where
$$\beta=\begin{cases}
1,& \mathrm{if}\  0\leqslant \alpha<\infty,\\
1-\frac{1}{\gamma}, & \mathrm{if} \ -1<\alpha<0,
\end{cases}$$
$\gamma$ is a number in $(1, -\frac{1}{\alpha})$ if $-1<\alpha<0$.
\end{lemma}

Now we are ready to prove  Theorem \ref{invertibility1}.

\begin{proof}  [Proof of Theorem \ref{invertibility1}]We will show that $(1)\Rightarrow (2)\Rightarrow (3)\Rightarrow (4)\Rightarrow (1)$. Without
loss of generality, we may assume that $0\leqslant \varphi \leqslant 1$.

$(1)\Rightarrow (2)$: This is trivial.

$(2)\Rightarrow (3)$: Suppose that $\widetilde{\varphi}$ is bounded below  by some positive constant $\delta$. By Lemma  \ref{reverse Carleson measure1}, it sufficient to show that
there exists a  $\delta'\in (0, 1)$ such that
$$ |G\cap K| \geqslant \delta' |\mathbb D \cap K|$$
for all  balls $K$ whose centers lie on $\partial \mathbb D$.

Since $\omega_{\alpha}dA$ is a rotation invariant measure, it is no loss of generality to assume that $K$ has its center at the point $(1, 0)$. It is also clear that we need only to prove the inequality for sufficient small radius $t$, say $t<1$.

Now we consider the subset $G=\{\lambda \in \mathbb D: \varphi(\lambda)>\frac{\delta}{4}\}$.  For each $z \in \mathbb D$,
\begin{align*}
\delta &\leqslant \widetilde{\varphi}(z)=\frac{1}{\|R^\alpha_z\|^2_{L^2(\omega_\alpha)}}\int_{\mathbb D}\varphi(\lambda)|R^\alpha_z (\lambda)|^2\omega_\alpha(\lambda) dA(\lambda)\\
&=\frac{1}{\|R^\alpha_z\|^2_{L^2(\omega_\alpha)}}\bigg(\int_{G}+\int_{\mathbb D \setminus G} \bigg) \varphi(\lambda)|R^\alpha_z( \lambda)|^2\omega_\alpha(\lambda) dA(\lambda)\\
&\leqslant (1-|z|^2)^{2+\alpha}\int_{G} \varphi(\lambda)|R^\alpha_z(\lambda)|^2\omega_\alpha(\lambda) dA(\lambda) +\frac{\delta}{4}.
\end{align*}
Let $L_z$ be the following integral:
\begin{align*}
&L_z=\frac{1}{\|R^\alpha_z\|^2_{L^2(\omega_\alpha)}}\int_{G\cap K} \varphi(\lambda)|R^\alpha_z(\lambda)|^2\omega_\alpha(\lambda)  dA(\lambda).
\end{align*}
Then for each $z\in \mathbb D$, we have
\begin{align*}
L_z&=\frac{1}{\|R^\alpha_z\|^2_{L^2(\omega_\alpha)}}\bigg(\int_{G}-\int_{G \backslash K}\bigg)\varphi(\lambda)|R^\alpha_z( \lambda)|^2\omega_\alpha(\lambda) dA(\lambda)\\
&\geqslant \frac{1}{2}(1-|z|^2)^{2+\alpha}\bigg(\int_{G}-\int_{G \backslash K}\bigg)\varphi(\lambda)|R^\alpha_z(\lambda)|^2\omega_\alpha(\lambda) dA(\lambda)\\
& \geqslant \frac{3\delta}{8}-\frac{1}{2}(1-|z|^2)^{2+\alpha} \int_{G \backslash K}\varphi(\lambda)|R^\alpha_z(\lambda)|^2\omega_\alpha(\lambda)  dA(\lambda)\\
&\geqslant \frac{3\delta}{8}-\frac{1}{2}(1-|z|^2)^{2+\alpha} \int_{G \backslash K}|R^\alpha_z(\lambda)|^2\omega_\alpha(\lambda)  dA(\lambda)\ \ \  \big(\mathrm{using}\ 0\leqslant \varphi \leqslant 1\big).
\end{align*}

For the $\delta$ above, Lemma \ref{D-K} guarantees that we can select  $z_0\in \mathbb D$ to define a function $f$ (as the one in Lemma \ref{D-K}) satisfies the following three inequalities:
\begin{align*}
L_{z_0}& \geqslant \frac{3\delta}{8}-\frac{1}{2}(1-|z_0|^2)^{2+\alpha} \int_{G \backslash K}|R^\alpha_{z_0}(\lambda)|^2\omega_\alpha(\lambda)  dA(\lambda)\\
&=\frac{3\delta}{8}-\frac{1}{2}\int_{G \backslash K} |f(\lambda)|^2(1-|\lambda|^2)^\alpha dA(\lambda),
\end{align*}
$$\int_{G \backslash K} |f(\lambda)|^2(1-|\lambda|^2)^\alpha dA(\lambda)<\frac{\delta}{4}$$
and
$$\int_{G\cap K}|f(\lambda)|^2 (1-|\lambda|)^\alpha dA(\lambda)\leqslant C \bigg(\frac{|G\cap K|}{|\mathbb D \cap K|}\bigg)^\beta,$$
where the constant $C$ depends only on $\delta$ and $\alpha$.  Therefore,
\begin{align*}
\frac{\delta}{4} &\leqslant L_{z_0} \leqslant (1-|z_0|^2)^{2+\alpha}\int_{K\cap G}|R^\alpha_{z_0}(\lambda)|^2\omega_\alpha(\lambda) dA(\lambda)\\
&=\int_{G\cap K}|f(\lambda)|^2 (1-|\lambda|^2)^\alpha dA(\lambda)\\
&\leqslant C \bigg(\frac{|G\cap K|}{|\mathbb D \cap K|}\bigg)^\beta.
\end{align*}
Now we get $(3)$ by Lemma \ref{reverse Carleson measure1}.

$(3)\Rightarrow (4)$: Observer that
\begin{align*}
\int_{\mathbb D} |f(z)|^2  \varphi(z)  \omega_{\alpha}(z) dA(z)&\geqslant \int_{G}\varphi(z) |f(z)|^2 \omega_{\alpha}(z) dA(z)\\
&> r\int_{G} |f(z)|^2 \omega_{\alpha}(z) dA(z)\\
&\geqslant \frac{r}{C} \int_{\mathbb D} |f(z)|^2 \omega_{\alpha} (z) dA(z),
\end{align*}
the last inequality follows from the definition of the reverse Carleson measure.

$(4)\Rightarrow (1)$: Use the same arguments as the proof of Corollary 3 in \cite{Lue}, we obtain that
$\|I-T_\varphi\|<1$, which implies that $T_\varphi$ is invertible on $L_h^2(\omega_{\alpha})$.  This completes the whole proof of the theorem.
\end{proof}

Let $\mathbb T_\varphi$  denote the Toeplitz operator with symbol $\varphi$ on the weighted Bergman space $L_a^2(\omega_\alpha)$.
 Combining the main result in \cite{Lue} and the techniques used in the proof of Theorem \ref{invertibility1}, we can generalize Theorem 3.2 in \cite{Zhao} to the case of standard weighted Bergman space.

\begin{thm}\label{invertibility2}
 Let $\varphi\geqslant 0$ be in $L^\infty(\omega_\alpha)$. Then the following  are equivalent:\\
$(1)$ The Toeplitz operator $\mathbb T_\varphi$  is invertible on $L_a^2(\omega_\alpha)$;\\
$(2)$ The Berezin transform $\widehat{\varphi}$ is invertible in $L^\infty(\omega_\alpha)$, where
$$\widehat{\varphi}(z):=\frac{1}{\|K^\alpha_z\|^2_{L^2(\omega_\alpha)}}\int_{\mathbb D}\varphi(\lambda)|K^\alpha_z( \lambda)|^2\omega_\alpha(\lambda)dA(\lambda);$$
$(3)$ There exists $r>0$ such that  $$G:=\{z \in \mathbb D: \varphi(z)>r \},$$ $\chi_{G}dA$ is a reverse Carleson measure for $L_a^2(\omega_\alpha)$;\\
$(4)$ There exists a  $\delta\in (0, 1)$ such that
$$ |G\cap K| \geqslant \delta |\mathbb D \cap K|$$
for every ball $K$ whose center lies on $\partial \mathbb D$;\\
$(5)$ There exists a constant $C>0$ such that
$$\int_ {\mathbb D}|f(z)|^2 \varphi(z) \omega_\alpha(z) dA(z)\geqslant C \int_{\mathbb D}|f(z)|^2\omega_\alpha(z) dA(z)$$
for $f\in L_a^2(\omega_\alpha)$.
\end{thm}
\begin{proof}
From the proof of Theorem \ref{invertibility1}, it sufficient to show that one can replace the harmonic function $f$ defined in Lemma
\ref{D-K} by a suitable analytic function $g$. Indeed, we  construct the desired function $g$ as follows. Suppose that $K$ has radius $0<t<1$  and center $u=(1, 0)\in \mathbb R^2$. Define
 $$g(\lambda)=\sqrt{\alpha+1}K^\alpha_{z_0}(\lambda)(1-|z_0|^2)^{\frac{2+\alpha}{2}},$$
where $z_0=(1-st)u$, $0<s<1$. Then it is not hard to check that the two inequalities in Lemma
\ref{D-K} both hold for $g$. Now the rest of the  proof parallels exactly one given in Theorem \ref{invertibility1}.
\end{proof}

Based on the theorem above, we can establish a relationship of the invertibility  between  Toeplitz operators with nonnegative symbols on $L_a^2(\omega_\alpha)$ and $L_h^2(\omega_\alpha)$.
\begin{cor}\label{invertibility3}
Let $\varphi$ be a nonnegative bounded measurable function on $\mathbb D$. The following four conditions are equivalent:\\
$(1)$  $\mathbb T_\varphi$  is invertible on $L_a^2(\omega_\alpha)$;\\
$(2)$  $T_\varphi$  is invertible on $L_h^2(\omega_\alpha)$;\\
$(3)$  $\widehat{\varphi}$ is invertible in $L^\infty(\omega_\alpha)$;\\
$(4)$  $\widetilde{\varphi}$ is invertible in $L^\infty(\omega_\alpha)$.
\end{cor}

 To end this section, we study Toeplitz operators with bounded analytic symbols on $L_h^2(\omega_\alpha)$. Unlike the  Bergman space  $L_a^2(\omega_\alpha)$ setting, the spectrum of analytic Toeplitz operator $T_\varphi$ on $L_h^2(\omega_\alpha)$ is not
equal to the closure of $\varphi(\mathbb D)$.  Let $H^\infty$ denote space of  analytic functions in $L^\infty(\omega_\alpha)$, we have the following result.
\begin{thm}\label{Analytic}
Suppose that $\varphi \in H^\infty$. If $\varphi$ is invertible, then the Toeplitz operator $T_\varphi$
is invertible on $L_h^2(\omega_\alpha)$. However, the converse is not true.
\end{thm}

Let $\varphi \in H^\infty$  be  invertible. However, $T_\varphi T_{1/\varphi}=T_{1/\varphi}T_\varphi$ does not hold on the harmonic Bergman space in the general case (see Theorem 5 in \cite{Choe}). For this reason, we need to find some relationships between $T_\varphi$ and $\mathbb T_\varphi$ to study the invertibility problem. To do so, let us introduce some notations first.

For $\varphi \in L^\infty(\omega_\alpha)$, we define $$\varphi^\star(z)=\varphi(\overline{z}) \ \  \mathrm{and}\ \
\varphi^*(z)=\overline{\varphi(\overline{z})}.$$

Next, we define the unitary operator $W: L_h^2(\omega_\alpha)=zL_a^2(\omega_\alpha)\oplus \overline{L_a^2(\omega_\alpha)}\rightarrow zL_a^2(\omega_\alpha)\oplus L_a^2(\omega_\alpha)$ by
$$W=\left( {\begin{array}{*{20}{c}}
   {{I}} & {{O}}  \\
   {{O}} & {{U}}  \\
\end{array}} \right),$$
where $U$ is a unitary operator on $L^2(\omega_\alpha)$, defined by $Uf(z)=f(\overline{z})$.

It is clear that $W^*$ maps $zL_a^2(\omega_\alpha)\oplus L_a^2(\omega_\alpha)$ to $L_h^2(\omega_\alpha)$ and
$$W^*=\left( {\begin{array}{*{20}{c}}
   {{I}} & {{O}}  \\
   {{O}} & {{U}}  \\
\end{array}} \right).$$

For $f$ and $g$ in $L^2(\omega_\alpha)$, let $f\otimes g$ be the rank one operator defined by
$$(f\otimes g)h=\langle h, g\rangle_{L^2(\omega_\alpha)}f$$
for $h\in L^2(\omega_\alpha)$.

Let $P$ denote the orthogonal projection from $L^2(\omega_\alpha)$ to $L_a^2(\omega_\alpha)$, we define the Hankel operator with symbol $\varphi$ acting on $L_a^2(\omega_\alpha)$ by $\mathbb H_{\varphi} f=P(\varphi U f)$.
Using the notations  above,  we obtain the following matrix representation of $WT_\varphi W^*$, which  will be used to study the invertibility problem.
\begin{lemma}\label{matrix representation}
Let $\varphi \in L^\infty(\omega_\alpha)$. On the space $zL_a^2(\omega_\alpha)\oplus L_a^2(\omega_\alpha)$ we have
$$W T_\varphi W^*=\left( {\begin{array}{*{20}{c}}
   {{\mathbb{T}_\varphi-(1\otimes \overline{\varphi})}} & {{\mathbb{H}_\varphi- (1\otimes \varphi^*)}}  \\
   {{\mathbb H_{\varphi^\star}}} & {{\mathbb {T}_{\varphi^\star}}}  \\
\end{array}} \right).$$
\end{lemma}
  \begin{proof}The proof is exactly the same as the proof of  Theorem 2.1 in \cite{Guo}, we omit the details here.
  \end{proof}

We are now ready to prove Theorem \ref{Analytic}. \\
\begin{proof} [Proof of Theorem \ref{Analytic}.] For $\varphi \in H^\infty$ and $f\in zL_a^2(\omega_\alpha)$ we have
$$\mathbb H_{\varphi^\star} f=P(\varphi^\star Uf)=P(\varphi(\overline{z})f(\overline{z}))=0$$
and $$(1\otimes \overline{\varphi}) f= \langle f, \overline{\varphi}\rangle_{L^2(\omega_\alpha)} 1=0,$$
since $f(0)=0$.

 Consequently, the matrix representation of $W T_\varphi W^*$ with $\varphi \in H^\infty$ is given by
$$W T_\varphi W^*=\left( {\begin{array}{*{20}{c}}
   {{\mathbb{T}_\varphi}} & {{\mathbb{H}_\varphi- (1\otimes \varphi^*)}}  \\
   {{O}} & {{\mathbb {T}_{\varphi^\star}}}  \\
\end{array}} \right).$$
Note that $\varphi \in H^\infty$ is invertible implies that $\varphi^\star$ is also invertible,  so that  $\mathbb T_{\varphi^\star}$ is invertible on $L_a^2(\omega_\alpha)$.

On the other hand, $\mathbb T_\varphi$ is invertible on $zL_a^2(\omega_\alpha)$ follows from  $\varphi$ is invertible. The above matrix representation tells  us that $W T_\varphi W^*$ is invertible on $zL_a^2(\omega_\alpha)\oplus L_a^2(\omega_\alpha)$. Thus $T_\varphi$ is invertible on $L_h^2(\omega_\alpha)$.

To show there exists a function $\varphi$ in $H^\infty$ such that $\varphi$ is not invertible in $L^\infty(\omega_\alpha)$ but $T_\varphi$
is invertible on $L_h^2(\omega_\alpha)$, we consider the function $\varphi(z)=z.$ It is easy to show that
$$\mathrm{Ker}(T_z)=\mathrm{Ker}(T^*_z)=\mathrm{Ker}(T_{\overline{z}})=\{0\}.$$
Observe that $T_z$ is Fredholm,  we conclude that $T_z$ is invertible on $L_h^2(\omega_\alpha)$. This finishes the proof.
\end{proof}

\section{A Reverse Carleson type Inequality for $L_h^2(\omega)$}
In the preceding section, we study the invertibility problem  of Toeplitz operators via the reverse Carleson measures for standard weighted harmonic Bergman spaces. In this section, we  establish a sufficient condition for $\chi_G dA$ to be a reverse Carleson measure for the space $L_h^2(\omega)$,  where $\omega\in \mathcal{A}_2$ and $G$ is a measurable set in $\mathbb D$.

For $a\in\dd$, $0<r<1$, recall that
\begin{equation*}
S(a,r)=\{z\in\dd:|z-a|<r(1-|a|)\}.
\end{equation*}
The main result in this section is Theorem \ref{the weight w}, which is a harmonic version of Theorem 3.9 in \cite{Luecking1985}.

\begin{thm}\label{the weight w}
Suppose that $G\subset\dd$  and  $\omega$ satisfies $\mathcal{A}_2$ condition. If there exist  $\delta\in(0,1)$ and $r\in(0,1)$ such that for all $a\in\dd$
\begin{equation*}
|G\cap S(a, r)|\geqslant \delta |S(a, r)|,
\end{equation*}
then there exists a positive constant $C=C(r, \delta)$ such that
$$ \int_{\mathbb D}|f(z)|^2\omega(z)dA(z)\leqslant C \int_{G}|f(z)|^2\omega(z) dA(z)$$
for all $f\in L^2_h(\omega)$.
\end{thm}

To prove the above theorem, we will adopt some ideas and techniques in \cite{Luecking1985}. Firstly, we need to introduce a new (weight) function $\omega^*$ and discuss some properties of $\omega^*$. In the rest of this section,  we use ``$r$" and ``$\delta$" to denote the numbers provided in Theorem \ref{the weight w}.

Now we define a positive  function $\omega^*$ on the open unit disk as follows:
$$\omega^*(z)=\omega_r^*(z):=\frac{|S(z, r)|_\omega}{|S(z, r)|}.$$
It is clear that $\omega^*\in L^1(dA)$, and so $\omega^*$ is a weight. Moreover, $\omega^*$ has
the following important property.
\begin{lemma}\label{measure estimate0}
Let $z\in \mathbb D$,  there exist constants $C_1$ and $C_2$ depending only on $r$ such that
$$C_1\omega^*(a)\leqslant \omega^*(z)\leqslant C_2 \omega^*(a)$$
for all  $a\in S(z, r)$. Consequently, we have
$$\int_{\mathbb D}\frac{\omega^*(a)}{|S(a, r)|} \chi_{S(a, r)}(z) dA(a)\leqslant C_3\omega^*(z)\ \ \ \big(z \in \mathbb D\big), $$
 where $C_3=C_3(r)$ is a constant.
\end{lemma}
\begin{proof}
By Lemma 2.2  in \cite{Con}, there exists a positive  constant $C$ depending only on $r$  such that
$$ C^{-1}|S(a, r)|_\omega \leqslant |S(z, r)|_\omega \leqslant C|S(a, r)|_\omega $$
 Moreover, it is well known that  $|S(z, r)|$ is equivalent to $|S(a, r)|$ (with constants independent of $a$ and $z$) if  $a\in S(z, r)$.
This gives the first conclusion of the lemma.  Based on this result,  we have
\begin{align*}
\int_{\mathbb D}\frac{\omega^*(a)}{|S(a, r)|}\chi_{S(a, r)}(z)dA(a)
&\leqslant C \int_{\mathbb D}\frac{\omega^*(z)}{|S(a, r)|}\chi_{S(a, r)}(z)dA(a)\ \ \  \big(\mathrm{using}\ z\in S(a,r)\big)\\
 &\leqslant C \int_{\mathbb D}\frac{\omega^*(z)}{|S(a, r)|}\chi_{D(a, r)}(z)dA(a) \ \ \  \big(\mathrm{since}\ S(a,r)\subset D(a,r)\big) \\
 &=C\int_{\mathbb D}\frac{\omega^*(z)}{|S(a, r)|}\chi_{D(z, r)}(a)dA(a)\\
 &=C\int_{D(z, r)}\frac{\omega^*(z)}{|S(a, r)|}dA(a)\\
& \leqslant C_3 \omega^*(z),
\end{align*}
where $C_3$ is a positive constant depending only on $r$, as desired.
\end{proof}
Another property of $\omega^*$ is given by the following inequality, which will be used to estimate the integral of $|f|^2 \omega$ over the subset $G$.
\begin{lemma}\label{norm of w less then norm of w*} Let $\omega$ be an $\mathcal {A}_2$ weight. Then there exists a constant $C>0$ depending only on  $r$ such that
$$\|f\|_{L^2(\omega)}^2 \leqslant C \|f\|_{L^2(\omega^*)}^2$$
for all $f\in L_h^2(\omega)$.
\end{lemma}
\begin{proof}
By definitions, we have
\begin{align*}
\|f\|_{L^2(\omega^*)}^2&=\int_{\mathbb D} |f(z)|^2\omega^*(z)dA(z)\\
&=\int_{\mathbb D} \omega(\xi)\bigg(\int_{\mathbb D} |f(z)|^2 \frac{\chi_{S(z, r)}(\xi)}{|S(z, r)|}dA(z)\bigg)dA(\xi),
\end{align*}
 We next deal with the bracketed expression. Observe
$$ S\Big(\xi, \frac{r}{2(1+r)}\Big)\subset \big\{z\in \mathbb D: |z-\xi|<r(1-|z|)\big\},$$
we obtain
\begin{align*}
\int_{\mathbb D} |f(z)|^2 \frac{\chi_{S(z, r)}(\xi)}{|S(z, r)|}dA(z)
&=\int_{\{z\in \mathbb D: |z-\xi|<r(1-|z|\}} \frac{|f(z)|^2 }{|S(z, r)|}dA(z)\\
&\geqslant \int_{S(\xi, \frac{r}{2(1+r)})} \frac{|f(z)|^2 }{|S(z, r)|}  dA(z)\\
& \geqslant \frac{C}{|S(\xi, r)|} \int_{S(\xi, \frac{r}{2(1+r)})} |f(z)|^2 dA(z) \ \ \  \ \bigg(\mathrm{since}\ z\in S\big(\xi,\frac{r}{2(1+r)}\big)\bigg)\\
&=C\bigg[\frac{1}{|S(\xi, \frac{r}{2(1+r)})|} \int_{S(\xi, \frac{r}{2(1+r)})} |f(z)|^2 dA(z)\bigg] \cdot \frac{|S(\xi, \frac{r}{2(1+r)})|}{|S(\xi, r)|}\\
&\geqslant \frac{C}{16} |f(\xi)|^2,
\end{align*}
the last inequality follows form the subharmonicity  of $|f|^2$ (see Lemma \ref{subharmonic behavior of harmonic functions}) and
$C$ depends only on $r$.  Thus we get
$$\|f\|_{L^2(\omega^*)}^2\geqslant \frac{C}{16}  \int_{\mathbb D} |f(\xi)|^2 \omega(\xi)dA(\xi)=\frac{C}{16} \|f\|^2_{L^2(\omega)},$$
which finishes the proof of Lemma \ref{norm of w less then norm of w*}.
\end{proof}

In order to complete the proof the main theorem in this section, the following two key lemmas are also needed.
\begin{prop}\label{the measure of F}
Let $G$ be the subset which satisfies the assumption in Theorem \ref{the weight w}. For $\eta\in (0, 1)$, we  define a subset $F$  as the following:
 $$F:=\{z\in \mathbb D: \omega(z)\geqslant \eta \omega^*(z)\}.$$
Then one can choose  $\eta$ (depending only on $\delta$ and $r$) sufficiently small such that
$$|F\cap S(a, r)|\geqslant (1-\frac{\delta}{2})|S(a, r)|$$
and
$$\Big| G\cap S(a, r)\cap F\Big|\geqslant \frac{\delta}{2}|S(a, r)|$$
for all $a\in \mathbb D$,

\end{prop}
\begin{proof} We first claim that for any $\delta'\in (0, 1)$,
there exists  $\eta'=\eta'(\delta')>0$ such that
$$\big|\{z\in S(a, r): \omega(z)<\eta' \omega^*(a)\}\big|<\delta' |S(a, r)|$$
for all $a\in \mathbb D$.

Indeed, for each $\kappa \in (0, 1)$ and $a\in \mathbb D$, we have
\begin{align*}
&\big|\{z\in S(a, r): \omega(z)<\kappa \omega^*(a)\}\big|\cdot \frac{1}{\kappa \omega^*(a)}\\
&< \int_{\{z\in S(a, r): \omega(z)<\kappa \omega^*(a)\}}\frac{1}{\omega(z)} dA(z)\\
&\leqslant |S(a, r)|_{\omega^{-1}}\\
&\leqslant [\omega]_{\mathcal{A}_2} |S(a, r)|^2\cdot |S(a, r)|_\omega^{-1}.
\end{align*}
So, we obtain
$$\big|\{z\in S(a, r): \omega(z)<\kappa\omega^*(a)\}\big| <  ([\omega]_{\mathcal{A}_2}  \kappa) |S(a, r)|$$
for all $a\in \mathbb D$ and $\kappa\in (0, 1)$.
By this inequality, we can choose any  $0<\eta' \leqslant [\omega]_{\mathcal{A}_2}^{-1}\delta'$ to finish the proof of the claim.

Lemma \ref{measure estimate0} guarantees  that there is a constant $C=C(r)$ such that
$$\{z\in S(a, r): \omega(z)<C \tau \omega^*(z)\}\subset \{z\in S(a, r): \omega(z)< \tau  \omega^*(a)\}\ \ \ \big( a\in \mathbb D \big) $$
for every $\tau\in (0, 1)$. By the claim and its proof, there exists a   $\tau=\tau(\delta)<C^{-1}$ such that
$$\big|\{z\in S(a, r): \omega(z)< \tau \omega^*(a)\}\big|<\frac{\delta}{2}|S(a, r)|  \ \ \ \big( a\in \mathbb D \big) .$$

Therefore, we can take  $\eta=\eta(\delta, r)=C\tau<1$ such that
$$\big|\{z\in S(a, r): \omega(z)< \eta \omega^*(z)\}\big|<\frac{\delta}{2}|S(a, r)|  \ \ \ \big( a\in \mathbb D \big).$$
Using this $\eta$ to define the corresponding  $F$, so that
$$|F\cap S(a, r)|=\big|\{z\in S(a, r): \omega(z)\geqslant \eta \omega^*(z)\}\big| \geqslant (1-\frac{\delta}{2})|S(a, r)|$$
for all $a\in \mathbb D$.

By assumption
$$|G\cap S(a, r)|\geqslant \delta | S(a, r)|,$$
we have
\begin{align*}
\delta |S(a, r)|&\leqslant |G\cap S(a, r)|\\
&=\Big| \big[G\cap S(a, r)\cap F\big]\cup \big[G\cap S(a, r)\cap (\mathbb D \setminus F)\big]\Big|\\
&\leqslant \Big| G\cap S(a, r)\cap F\Big|+|S(a, r)\cap (\mathbb D \setminus F)|\\
&= \Big| G\cap S(a, r)\cap F\Big|+|S(a, r)|-|S(a, r)\cap F|\\
&\leqslant \Big| G\cap S(a, r)\cap F\Big|+|S(a, r)|-(1-\frac{\delta}{2})|S(a, r)|.
\end{align*}
This gives that
$$\Big| G\cap S(a, r)\cap F\Big|\geqslant \frac{\delta}{2}|S(a, r)|$$
for all $a\in \mathbb D$, as desired.
\end{proof}

\begin{lemma}\label{the weight function w*} If  $G_0$ is  a measurable  subset of $\mathbb D$  that satisfies
 $$|G_0\cap S(a, r)|\geqslant \delta_0 |S(a, r)| \ \ \  \big( a \in \mathbb D \big)$$
for some $\delta_0>0$,
then
there exists constant $C=C(\delta_0, r)>0$ such that
$$ \int_{\mathbb D}|f(z)|^2\omega^*(z)dA(z)\leqslant C \int_{G_0}|f(z)|^2\omega^*(z) dA(z)$$
for all $f\in L_h^2(\omega^*)$.
\end{lemma}

The proof of the preceding lemma is some what long and it requires a number of technical lemmas. Let us assume this result  for the moment and we will prove it at the end of this section. Now we give the proof of Theorem \ref{the weight w}. \vspace{2mm}\\
\begin{proof}[Proof of Theorem \ref{the weight w}]  By  Proposition \ref{the measure of F} and Lemma \ref{the weight function w*},
 we have
$$ \int_{\mathbb D}|f(z)|^2\omega^*(z)dA(z)\leqslant C_1 \int_{G\cap F}|f(z)|^2\omega^*(z) dA(z)\leqslant C_1 \eta^{-1}\int_{G}|f|^2\omega dA$$
for all $f \in L_h^2(\omega^*)$, where $C_1$ is a constant depending only on $r$ and $\eta=\eta(\delta, r)<1$ is chosen by Proposition \ref{the measure of F}.

From Lemma \ref{norm of w less then norm of w*}, it is clear that
$$ \int_{\mathbb D}|f(z)|^2\omega (z)dA(z)\leqslant C_1 \eta^{-1} \int_{G}|f(z)|^2\omega (z) dA(z)$$
for all $f\in L_h^2(\omega)$,
which gives the desired inequality in  Theorem \ref{the weight w}.
\end{proof}

Now we turn to the proof of Lemma \ref{the weight function w*}.  Before giving the proof, we need to introduce some notations and prove three technical lemmas.

Let $0<\theta<\frac{1}{2}$, we define the subset
$$E_\theta(a)=E_\theta(f, a):=\Big\{z\in S(a, r): |f(z)| >\theta |f(a)| \Big\}$$
and  the following operator:
$$B_\theta f(a):=\frac{1}{|E_\theta(a)|}\int_{E_\theta(a)}|f(z)|^2 dA(z)\ \ \ \big(a\in \mathbb D \big).$$
It is clear that
$$B_\theta f(a)\geqslant \frac{1}{|S(a, r)|}\int_{S(a, r)}|f(z)|^2 dA(z) \ \ \ \big(a\in \mathbb D \big).$$

For $\epsilon \in (0, 1)$, we consider the  following two subsets, which are very useful to establish our main result.
Define $$A=A_\epsilon:=\bigg\{a\in\mathbb D: |f(a)|^2 \leqslant \frac{\epsilon}{|S(a, r)|}\int_{S(a, r)}|f(z)|^2 dA(z)\bigg\}$$
and
$$B=B_\epsilon:=\bigg\{a\in\mathbb D: |f(a)|^2 \leqslant \epsilon^2 B_\theta f(a)\bigg\}.$$

A useful estimation for the Lebesgue measure of $\big\{z\in S(a, r): |f(z)|>\theta|f(a)|\big\}$ with $f$  harmonic is the following inequality.
\begin{lemma}\label{measure estimate3}
Fix $\epsilon\in (0, 1)$. For any $\delta'\in (0,1 )$, there exists $ \theta \in (0, \frac{1}{2})$ such that
$$\bigg|\Big\{z\in S(a, r): |f(z)|> \theta |f(a)| \Big\}\bigg|>\big(1-\frac{\delta'}{2}\big)|S(a, r)|$$
for every   $f$  harmonic on $\mathbb D$ and  satisfies
$$|f(a)|^2> \frac{\epsilon^2}{|S(a, r)|}\int_{S(a, r)}|f(z)|^2dA(z)\ \ \ (a\in \mathbb D).$$
\end{lemma}
\begin{proof}
See the proof of Lemma 2 in  \cite{Luecking1983}.
\end{proof}
 The next lemma provides  an estimation for the integral of $|f|^2\omega^*$ over the set $A$.
\begin{lemma}\label{the integral over A}
Let $\epsilon \in (0, 1)$, then there exists a constant
$C$ (independent of $\epsilon$) such that
$$\int_{A}|f(z)|^2\omega^*(z)dA(z)\leqslant C\epsilon \int_{\mathbb D}|f(z)|^2\omega^*(z)dA(z)$$
for all $f\in L_h^2(\omega^*)$.
\end{lemma}
\begin{proof} For $a\in A$,  we have
$$|f(a)|^2 \leqslant \frac{\epsilon}{|S(a, r)|}\int_{\mathbb D}\chi_{S(a, r)}(z)|f(z)|^2 dA(z).$$
Multiplying the above inequality by $\omega^*(a)$ and integrating over $A$  to obtain
\begin{align*}
\int_{A}|f(a)|^2 \omega^*(a) dA(a) & \leqslant \int_{A}\frac{\epsilon}{|S(a, r)|}\int_{\mathbb D}\chi_{S(a, r)}(z)|f(z)|^2 dA(z)\omega^*(a)dA(a)\\
&=\epsilon \int_{\mathbb D} |f(z)|^2 \bigg(\int_A \frac{\chi_{S(a, r)}(z)}{|S(a, r)|}\omega^*(a) dA(a)\bigg) dA(z).
\end{align*}
By the second conclusion of Lemma \ref{measure estimate0}, we have
\begin{align*}
\int_A \frac{\chi_{S(a, r)}(z)}{|S(a, r)|} \omega^*(a) dA(a)& \leqslant C\omega^*(z),
\end{align*}
where $C$  depends only on $r$. This completes the proof.
\end{proof}

The proof of Lemma \ref{the weight function w*} requires the following inequality, which can be proved by the preceding  lemma.
\begin{lemma}\label{the integral over B}
Let $\epsilon \in (0, 1)$, then there exists a constant
$C=C(r)$ such that
$$\int_{B}|f(z)|^2\omega^*(z)dA(z)\leqslant C\epsilon \int_{\mathbb D}|f(z)|^2\omega^*(z)dA(z)$$
for all $f\in L_h^2(\omega^*)$.
\end{lemma}
\begin{proof} Observe that
\begin{align*}
\int_{B}|f(z)|^2\omega^*(z)dA(z)&=\int_{B\cap A}|f(z)|^2\omega^*(z)dA(z)+\int_{B\setminus A}|f(z)|^2\omega^*(z)dA(z)\\
&\leqslant \int_{ A}|f(z)|^2\omega^*(z)dA(z)+\int_{B\setminus A}|f(z)|^2\omega^*(z)dA(z).
\end{align*}
Based on Lemma \ref{the integral over A}, it sufficient for us to show the following  inequality holds for some constant $C=C(r)$:
$$J:=\int_{B\setminus A}|f(z)|^2\omega^*(z)dA(z)\leqslant C\epsilon \int_{\mathbb D}|f(z)|^2\omega^*(z)dA(z).$$
Recall that for each $a\in B$ we have
$$|f(a)|^2 \leqslant \frac{\epsilon^2}{|E_\theta(a)|}\int_{E_\theta (a)}|f(z)|^2dA(z).$$
From the above inequality we have
\begin{align*}
J&=\int_{B\setminus A}|f(a)|^2\omega^*(a)dA(a)\\
&\leqslant \epsilon^2 \int_{B \setminus  A} \bigg( \frac{1}{|E_\theta(a)|}\int_{E_\theta(a)}|f(z)|^2dA(z)\bigg) \omega^*(a)dA(a)\\
&= \epsilon^2 \int_{\mathbb D} \bigg( \int_{B\setminus A} \frac{\omega^*(a)}{|E_\theta(a)|} \chi_{E_{\theta}(a)}(z) dA(a)\bigg)|f(z)|^2dA(z)\\
&\leqslant  \epsilon^2 \int_{\mathbb D} |f(z)|^2 \bigg( \int_{B \setminus A} \omega^*(a) \frac{\chi_{S(a, r)}(z)}{|E_\theta(a)|} dA(a)\bigg) dA(z).
\end{align*}
The last ``$\leqslant $" is due to  $E_\theta(a) \subset S(a, r)$.

To finish the proof, we need the following claim.\\
\textbf{Claim.} There is a positive constant $C=C(r)$ such that
$$|E_\theta(a)|\geqslant C \epsilon |S(a, r)|\ \ \ \mathrm{or}\ \ \ |E_\theta(a)|\geqslant C |S(a, r)|$$
for each $a\notin A.$

If the above claim is true, then we get
\begin{align*}
\int_{B \setminus A} \omega^*(a) \frac{\chi_{S(a, r)}(z)}{|E_\theta(a)|} dA(a)
& \leqslant C^{-1} \epsilon \int_{B \setminus A}\omega^*(a) \frac{\chi_{S(a, r)}(z)}{|S(a, r)|}dA(a).
\end{align*}
Use Lemma \ref{measure estimate0} again, we have
\begin{align*}
\int_{B \setminus A}\omega^*(a) \frac{\chi_{S(a, r)}(z)}{|S(a, r)|}dA(a) \leqslant C_1 \omega^*(z),
\end{align*}
the constant $C_1$ depends only on $r$. From the definition of $J$, we obtain
$$J\leqslant C\epsilon \int_{\mathbb D}|f(z)|^2\omega^*(z)dA(z)$$
for some positive constant $C=C(r)$.

Now we turn to prove the claim.  For each $a\notin A$, we have
\begin{align*}
|f(a)|^2 &>\frac{\epsilon}{|S(a, r)|}\int_{S(a, r)}|f(z)|^2 dA(z)\\
&=\frac{\epsilon}{r^2(1-|a|)^2}  \int_{S(a, r)} |f(z)|^2dA(z).
\end{align*}
Using the change of variables $z=a+r(1-|a|)\lambda$ to get
\begin{align*}
|f(a)|^2 &> \epsilon \int_{\mathbb D} |f(a+r(1-|a|)\lambda)|^2dA(\lambda).
\end{align*}
Let $g(\lambda):=f(a+r(1-|a|)\lambda)$, then $g$ is also harmonic on $\mathbb D$ and
$$|g(0)|^2> \epsilon \int_{\mathbb D} |g(\lambda)|^2dA(\lambda).$$

Applying Lemma \ref{gradient estimation} to the function $g$ to get a constant $C_0=C_0(r)$ such that
\begin{align*}
|g(z)-g(0)| & \leqslant C_0 |z| \int_{D(0, \frac{r}{4})} |g(\lambda)| dA(\lambda)\leqslant C_0 |z| \int_{\mathbb D} |g| dA
\end{align*}
whenever $|z|\leqslant \frac{r}{16}.$  Cauchy-Schwartz inequality gives that
$$|g(z)-g(0)| \leqslant C_0 |z| \bigg(\int_{\mathbb D} |g(\lambda)|^2 dA(\lambda)\bigg)^{\frac{1}{2}}\leqslant C_0 \epsilon^{-\frac{1}{2}}|g(0)|\cdot |z| .$$
provided $|z|\leqslant \frac{r}{16}$.

If
$$|z| < \min\Big\{\frac{r}{16}, \frac{\epsilon^{\frac{1}{2}}}{2C_0}\Big\},$$
then
$$|g(z)|\geqslant |g(0)|-|g(z)-g(0)|\geqslant \frac{|g(0)|}{2}.$$
Recall that $0<\theta< \frac{1}{2}$, we have
$|g(z)|>\theta |g(0)|$ for $|z|<\min\big\{\frac{r}{16}, \frac{\epsilon^{\frac{1}{2}}}{2C_0}\big\}.$ This means
that $$B\big(0, \frac{r}{16}\big) \subset \Big\{z\in \mathbb D: |g(z)|>\theta |g(0)| \Big \}$$
or
$$ B\big(0, \frac{\epsilon^{\frac{1}{2}}}{2C_0}\big) \subset \Big\{z\in \mathbb D: |g(z)|>\theta |g(0)| \Big \}.$$

On the other hand, observe that
\begin{align*}
|E_\theta(a)|&=\int_{\{z\in S(a, r): |f(z)|>\theta |f(a)|\}} dA(z)\\
&=\int_{\big\{|\frac{z-a}{r(1-|a|)}|<1: |f(z)|>\theta |f(a)|\big\}} dA(z)\\
&=|S(a, r)|\int_{\big\{|\lambda|<1: |f(a+r(1-|a|)\lambda)|>\theta |f(a)|\big\}} dA(\lambda)\\
&=|S(a, r)| \int_{\{|\lambda|<1: |g(\lambda)|>\theta |g(0)|\}} dA(\lambda)\\
&=|S(a, r)|\cdot \Big| \{\lambda \in \mathbb D: |g(\lambda)|>\theta |g(0)|\}\Big|.
\end{align*}
Therefore, we obtain
$$|E_\theta(a)|\geqslant \Big(\frac{r}{16}\Big)^2 |S(a, r)| $$
or
$$|E_\theta(a)|\geqslant \frac{\epsilon}{4C_0^2} |S(a, r)|.$$
This gives the proof of the claim, and so we complete the proof of Lemma \ref{the integral over B}.
\end{proof}

We are now  ready to prove Lemma \ref{the weight function w*}.
\begin{proof}[Proof of  Lemma \ref{the weight function w*}] Suppose that $|G_0\cap S(a, r)|\geqslant \delta_0 |S(a, r)|$.
From Lemmas \ref{the integral over A} and \ref{the integral over B}, we can choose
$\epsilon$ small enough so that
$$\int_{\mathbb D}|f(z)|^2\omega^*(z)dA(z)<2\int_{\mathbb D \setminus B}|f(z)|^2\omega^*(z)dA(z).$$
On the other hand, if $a\notin B$, then
\begin{align*}
|f(a)|^2 & >\epsilon^2 B_\theta f(a)\geqslant \frac{\epsilon^2}{|S(a, r)|}\int_{S(a, r)}|f(z)|^2dA(z).
\end{align*}
For the $\delta_0$ above, we  apply Lemma \ref{measure estimate3} to choose a $\theta\in (0, \frac{1}{2})$ such that
$$\bigg|\Big\{z\in S(a, r): |f(z)|> \theta |f(a)| \Big\}\bigg|>(1-\frac{\delta_0}{2})|S(a, r)|.$$
Since  $|G_0\cap S(a, r)|\geqslant \delta_0 |S(a, r)|$, we have
$$\bigg|\Big\{z\in S(a, r)\cap G_0: |f(z)|> \theta |f(a)| \Big\}\bigg|>\frac{\delta_0}{2}|S(a, r)|,$$
and so
$$\frac{1}{|S(a, r)|}\int_{S(a, r)\cap G_0}|f(z)|^2dA(z)>\frac{\theta^2 \delta_0}{2} |f(a)|^2\ \ \  \big(a\notin B\big).$$
Multiplying the above inequality by $\omega^*(a)$ and integrating over $\mathbb D \backslash B$  to get
\begin{align*}
\frac{\theta^2 \delta_0}{2} \int_{\mathbb D \backslash B} \omega^*(a) |f(a)|^2 dA(a)&<\int_{\mathbb D \backslash B}\frac{\omega^*(a)}{|S(a, r)|}\int_{S(a, r)\cap G_0}|f(z)|^2dA(z) dA(a)\\
&=\int_{G_0}|f(z)|^2\bigg(\int_{\mathbb D \backslash B} \frac{ \omega^*(a)}{|S(a, r)|} \chi_{S(a, r)}(z) dA(a)\bigg)dA(z)\\
&\leqslant C \int_{G_0}|f(z)|^2\omega^*(z) dA(z),
\end{align*}
the last ``$\leqslant$" follows from Lemma \ref{measure estimate0} and $C$ depends only on $r$.

Therefore,
  $$\int_{G_0}|f(z)|^2\omega^*(z) dA(z) > \frac{\delta_0 \theta^2}{4C} \int_{\mathbb D}|f(z)|^2\omega^*(z)dA(z)$$
  for all $f\in L_h^2(\omega^*)$. This completes the proof of Lemma \ref{the weight function w*}.
\end{proof}

\textbf{Acknowledgments} \   The authors are grateful to  Professor Kunyu Guo for helpful suggestions. Both authors
 would like to express their  deepest gratitude to Professor Dechao Zheng, X. Zhao's Ph.D thesis adviser, for his valuable guidance, encouragement and  support  in many years. \vspace{3mm}\\


\begin{thebibliography}{99}
{\footnotesize
\bibitem{AV2014}
T. Anderson, A. Vagharshakyan, A simple proof of the sharp weighted estimate for
Calderon-Zygmund operators on homogeneous spaces, \emph{Journal of Geometric Analysis}, 2014, 24(3): 1276-1297.

\bibitem{Cha}
R. Chac\'{o}n, Toeplitz Operators on Weighted Bergman Spaces, \emph{Journal of Function Spaces and Applications}, 2013, 2013: 1-5.

\bibitem{Choe}
B. Choe, Y. Lee, Commuting Toeplitz operators on the harmonic Bergman space, \emph{Michigan Mathematical Journal}, 1999, 46(1): 163-174.



\bibitem{Choi}
S. Choi, Positive Toeplitz operators on pluriharmonic Bergman spaces,  \emph{Journal of Mathematics of Kyoto University}, 2007, 47(2): 247-267.





\bibitem{Con2007}
O. Constantin, Discretizations of integral operators and atomic decompositions in vector-valued weighted Bergman spaces, \emph{Integral Equations and Operator Theory}, 2007, 59(4): 523-554.

\bibitem{Con}
O. Constantin, Carleson embeddings and some classes of operators on weighted Bergman spaces,  \emph{Journal of Mathematical Analysis and Applications}, 2010, 365(2): 668-682.



\bibitem{DW2011}
R. Douglas,  K. Wang, A harmonic analysis approach to essential normality of principal submodules,  \emph{Journal of Functional Analysis},
2011, 261(11): 3155-3180.


\bibitem{Guo}
K. Guo, D. Zheng, Toeplitz algebra and Hankel algebra on the harmonic Bergman space,  \emph{Journal of Mathematical Analysis and Applications}, 2002, 276(1): 213-230.




\bibitem{Hy2012}
T. Hyt\"{o}nen, The sharp weighted bound for general Calder\'{o}n-Zygmund operators,  \emph{Annals of Mathematics}, 2012, 175(3): 1473-1506.

\bibitem{Kuran}
\"{U}. Kuran, Subharmonic Behaviour of $|h|^p$ ($p>0$, $h$ harmonic), \emph{Journal of the London Mathematical Society}, 1974, 2(3): 529-538.

\bibitem{Lue}
D. Luecking,  Inequalities on Bergman spaces, \emph{Illinois Journal of Mathematics},  1981, 25(1): 1-11.

\bibitem{Luecking1983}
D. Luecking,  Equivalent norms on $L^p$ spaces of harmonic functions,  \emph{Monatshefte f\"{u}r Mathematik}, 1983, 96(2): 133-141.

\bibitem{Luecking1985}
D. Luecking, Representation and duality in weighted spaces of analytic functions,  \emph{Indiana University Mathematics Journal}, 1985, 34(2): 319-336.

\bibitem{Luecking AJM}
D. Luecking, Forward and reverse Carleson inequalities for functions in Bergman spaces and their derivatives, \emph{ American Journal of Mathematics}, 1985, 107(1): 85-111.





\bibitem{Miao}
J. Miao,  Toeplitz operators on harmonic Bergman spaces, \emph{Integral Equations and Operator Theory}, 1997, 27(4): 426-438.


\bibitem{Miao1998}
J. Miao,  Reproducing kernels for harmonic Bergman spaces of the unit ball,  \emph{Monatshefte f\"{u}r Mathematik}, 1998, 125(1): 25-35.

\bibitem{MW2014}
M. Mitkovski,  B. Wick, A reproducing kernel thesis for operators on Bergman-type function spaces, \emph{ Journal of  Functional  Analysis},  2014, 267(7): 2028-2055.

\bibitem{MW1974}
B. Muckenhoupt, R. Wheeden, Weighted norm inequalities for fractional
integrals, \emph{Transactions of the American Mathematical Society}, 1974, 192: 261-274.




%
%
\bibitem{PR2015}
J. \'{A}. Pel\'{a}ez, J. R\"{a}tty\"{a},  Embedding theorems for Bergman spaces via harmonic analysis, \emph{Mathematische Annalen},  2015, 36(1-2): 205-239.

\bibitem{PR2016}
J. \'{A}. Pel\'{a}ez, J. R\"{a}tty\"{a}, Trace class criteria for Toeplitz and composition operators on small Bergman spaces, \emph{Advances in Mathematics}, 2016, 293: 606-643.




\bibitem{Shu}
 Y. Shu, X. Zhao,  Positivity of Toeplitz operators on the harmonic Bergman spaces via Berezin transform,
   \emph{Acta Mathematica Sinica (English Series)}, 2016, 32(2): 175-186.

\bibitem{Sledd}
W. Sledd,  A note on $L^p$ spaces of harmonic functions, \emph{Monatshefte f\"{u}r Mathematik}, 1988, 106(1): 65-73.



\bibitem{Zhao}
X. Zhao, D. Zheng, Invertibility of Toeplitz operators via Berezin transforms, to appear in \emph{Journal of Operator Theory}.


\bibitem{Zhu2007}
K. Zhu,  Operator Theory in Function Spaces, 2nd ed,  Mathematical Surveys and Monographs, 138,  American Mathematical Society, 2007.























}
\end{thebibliography}
\end{document}